\documentclass[12pt]{article}

\usepackage{amsthm}
\usepackage{amssymb}
\usepackage{amsmath}
\usepackage{url} 
\usepackage[noabbrev,capitalise]{cleveref}

\usepackage{thm-restate}

\usepackage{color}
\usepackage[normalem]{ulem}

\usepackage[margin=1in]{geometry}

\newcommand{\eps}{\varepsilon}
\newcommand{\mc}[1]{\mathcal{#1}}

\newcommand{\mf}[1]{\mathfrak{#1}}
\newcommand{\bb}[1]{\mathbb{#1}}
\newcommand{\brm}[1]{\operatorname{#1}}

\theoremstyle{definition}

\theoremstyle{plain}
\newtheorem{thm}{Theorem}[section]
\newtheorem{lem}[thm]{Lemma}

\newtheorem{cor}[thm]{Corollary}

\newtheorem{obs}[thm]{Observation}

\newcommand{\h}{h}
\newcommand{\Forb}{\operatorname{Forb}}

\newcommand{\wpn}{\chi_c}
\newcommand{\ext}{\operatorname{ext}}
\newcommand{\crit}{\operatorname{dang}}
\newcommand{\dang}{\operatorname{dang}}
\newcommand{\red}{\operatorname{red}}

\title{Typical structure of hereditary graph families. I. Apex-free families}
\author{
Sergey Norin\thanks{Department of Mathematics and Statistics, McGill University. Email: {\tt sergey.norin@mcgill.ca}. Supported by an NSERC Discovery grant.}\and \and
Yelena Yuditsky\thanks{ Department of Mathematics, Ben-Gurion University of the Negev. Email: {\tt yuditskyl@gmail.com}.}}

\begin{document}

\maketitle

\begin{abstract} A family of graphs $\mc{F}$ is \emph{hereditary} if $\mc{F}$ is closed under isomorphism and taking induced subgraphs. The \emph{speed} of $\mc{F}$ is the sequence $\{|\mc{F}^n|\}_{n \in \bb{N}}$, where $\mc{F}^n$ denotes the set of graphs in $\mc{F}$ with the vertex set $[n]$.
Alon, Balogh, Bollob\'{a}s and Morris~\cite{ABBM11} gave a rough description of typical graphs in a hereditary family and used it to show for every proper hereditary family $\mc{F}$ there exist  $\eps>0$ and an integer $l \geq 1$ such that  $$|\mc{F}^n| =  2^{(1-1/l)n^2/2+o(n^{2-\eps})}.$$ 
The main result of this paper gives a more precise description of typical structure for a restricted class of hereditary families. As a consequence we characterize hereditary families with the speed just above the threshold $2^{(1-1/l)n^2/2}$, generalizing a result of Balogh and Butterfield~\cite{BB11}.  
	
\end{abstract}

\section{Introduction}

Let $\mc{F}$ be a  family of graphs. We say that $\mc{F}$ is \emph{hereditary} if $\mc{F}$ is closed under isomorphism and taking induced subgraphs. If $\mc{F}$ is additionally closed under taking arbitrary subgraphs, then we say that $\mc{F}$ is \emph{monotone}.

Let $[n]=\{1,2,\ldots,n\}$.
For a family of graphs $\mc{F}$, let $\mc{F}^n$ denote the set of graphs $G \in \mc{F}$ with $|V(G)|=[n]$. The \emph{speed} of $\mc{F}$ is a sequence $\{|\mc{F}^n|\}_{n \in \bb{N}}$, naturally measuring the size of  $\mc{F}$. The speed of monotone and hereditary families of graphs has been extensively studied since it has been first considered by Erd\H{o}s, Kleitman and Rotschild~\cite{EKR76}.

The key role in determining the speed of a hereditary family is played by its coloring number, which we now define.
Given integers $s,t \geq 0$, let $\mc{H}(s,t)$ denote the family of all graphs $G$ such that $V(G)$ admits a partition $(X_1,\ldots,X_s,Y_1,\ldots,Y_t)$, where $X_i$ is an independent set in $G$ for every $i \in [s]$, and $Y_i$ is a clique in $G$ for every $i \in[t]$. (Thus $\mc{H}(s,0)$ is a family of all $s$-colorable graphs. We write $\mc{H}(s)$ instead of  $\mc{H}(s,0)$, for brevity.) 

The \emph{coloring number} $\chi_c(\mc{F})$ of a graph family $\mc{F}$ is the maximum integer $l$ such that $\mc{H}(s,l-s) \subseteq \mc{F}$ for some $0 \leq s \leq l$. 

Sharpening the results of Alekseev~\cite{Alekseev92}, and Bollob\'{a}s and Thomason~\cite{BolTho95}, Alon,  Balogh, Bollob\'{a}s and Morris~\cite{ABBM11} gave the following, essentially tight characterization of possible speeds of  hereditary graph families.

\begin{thm}[\cite{ABBM11}]\label{t:AlonSpeed}
	Let $\mc{F}$ be a hereditary graph family, and let $l=\chi_c(\mc{F}) \geq 1$.  Then there exists $\eps >0$ such that
	$$2^{\left(1 - \frac{1}{l} \right)n^2/2-o(1)} \leq |\mc{F}^n| \leq 2^{\left(1 - \frac{1}{l} \right)n^2/2+o(n^{2 - \eps})}$$
	for all positive integers $n$.
\end{thm}

We say that a hereditary family $\mc{F}$ is \emph{thin} if $\chi_c(\mc{F}) \leq 1$. 
Let $\h(\mc{F},n)=\log(|\mc{F}^n|)$.\footnote{All logarithms in this paper are base $2$.}  By \cref{t:AlonSpeed} we have $\h(\mc{F},n) = O(n^{2 - \eps})$ for every thin $\mc{F}$.
 Schneiderman and Zito~\cite{SZ94} have shown that if the speed of a thin hereditary family is subexponential then it is severely constrained. 

\begin{thm}[\cite{SZ94}]\label{t:SZSpeed}
	Let $\mc{F}$ be a hereditary graph family. If  $\h(\mc{F},n)= o(n)$ then there exists a non-negative integer $k$ such that
$$\h(\mc{F},n) = k\log n + O(1).$$
\end{thm}

The proof of \cref{t:AlonSpeed} is based on a description of the typical structure of graphs in hereditary families, given ins~\cite{ABBM11}. We  present this description in Section~\ref{s:ABBM} and use as our main tool. 
The main result of this paper is \cref{t:critical}, which gives a much more precise description of this structure at the cost of imposing several technical restrictions on the hereditary family. As one of the corollaries of \cref{t:critical} we prove the following extension of \cref{t:SZSpeed} to general hereditary families. 

\begin{thm}\label{t:Speed}
 Let $\mc{F}$ be a hereditary graph family, and let $l=\chi_c(\mc{F})$. If  $\h(\mc{F},n)-\h(\mc{H}(l),n)= o(n)$ then there exists a non-negative integer $k$ such that
 $$\h(\mc{F},n) =\h(\mc{H}(l),n) + k\log n + O(1).$$
\end{thm}

\cref{t:SZSpeed} was sharpened by Balogh, Bollob\'{a}s and Weinreich~\cite{BBW00} who gave an exact structural description of thin hereditary families with subexponential speeds. \cref{t:Speed} is derived from  
a similar structural description, given in \cref{t:critical2}, which generalizes a theorem of Balogh and Butterfield~\cite{BB11}. 

In a sequel to this paper~\cite{NorYud192} we give further applications of
Theorems~\ref{t:critical} and~\ref{t:critical2}, where these theorems are used in conjunction with randomized constructions to provide examples of hereditary families with exotic typical behavior. These families provide an answer a question of Kang et al.~\cite{KMRS14} and Loebl et al.~\cite{LRSTT10} and give a counterexample to a conjecture of Reed and Scott~\cite{RS18}.

The rest of the paper is structured as follows. In Section~\ref{s:prelim} we introduce the necessary terminology and tools from the literature and state our main structural results.  In Section~\ref{s:structureproof} we  use the language of Kolmogorov complexity to develop a number of extra tools for analyzing typical structure of hereditary families, and prove \cref{t:critical}.  Finally, in Section~\ref{s:applications}, we derive ~\cref{t:critical2}, and thus ~\cref{t:Speed}, from \cref{t:critical}.

\section{Structure theorems for hereditary graph families}\label{s:prelim}

In this section we present additional definitions and  tools from the literature, as well as state our main structural results.

\subsection{Terminology and notation}\label{s:notation}

Let $H$ be a graph, we say that a graph $G$ is \textit{$H$-free} if $G$ does not contain an induced subgraph isomorphic to $H$. We denote by $\brm{Forb}(H)$ the family of all $H$-free graphs. Clearly, $\brm{Forb}(H)$ is hereditary. More generally, if $\mc{H}$ is a collection of graphs we say that a graph $G$ is \textit{$\mc{H}$-free} if $G$ is $H$-free for every $H \in \mc{H}$, and we denote by $\brm{Forb}(\mc{H})$ the family of all $\mc{H}$-free graphs. We say that a hereditary family $\mc{F}$ is \emph{finitely generated} if there exists a finite collection of graphs $\mc{H}$ such that $\mc{F} = \Forb(\mc{H})$. 

For a graph $G$ and $X \subseteq V(G)$ we denote by $G[X]$ the subgraph of $G$ induced by $X$. 
We say that an injection $\eta: V(H) \to V(G)$  is \emph{an embedding of a graph $H$ into $G$} if $\eta$ is an isomorphism between $H$ and $G[\eta(V(H))]$. Thus $G$ is $H$-free if and only if there does not exist an  embedding of $H$ into $G$.

An \emph{$(a,b)$-bigraph} (or simply a \emph{bigraph}) $J$ is an ordered triple $(H,A,B)$, where $H$ is a bipartite graph and $(A,B)$ is a bipartition of $H$ with $|A|=a$, $|B|=b$. If $A$ and $B$ are ordered then we say that $J$ is \emph{an ordered bigraph}.\footnote{It will be frequently convenient for us to assume that the graphs and bigraphs that we consider are ordered, without necessarily explicitly stating so.} Extending the above definitions from  graphs to bigraphs, we say that an injection $\eta: A \cup B \to V(G)$  is \emph{an embedding of a bigraph $J$ into a graph $G$} if for all $u \in A$ and $v \in B$ we have $uv \in E(H)$ if and only if $\eta(u)\eta(v) \in E(G)$. We say that a graph $G$ is \emph{$J$-free} if there exists no embedding of $J$ into $G$, and we denote by $\Forb(J)$ the hereditary family of all $J$-free graphs. Given two (ordered) subsets $A,B$ of $V(G)$ we denote by $G[A,B]$ the (ordered) bigraph induced by $(A,B)$ in $G$. 

\vskip 10pt

Informally, most of the results about typical structure of hereditary family state that for almost all graphs in a given family their vertex sets can be partitioned into constant number of pieces such that the graphs induced by the sets are ``structured". The following definitions are used to formalize such statements. 

For a graph family $\mc{F}$, we say that  a partition $\mc{X}=(X_1,\ldots, X_l)$  of $V(G)$ is an \emph{$(\mc{F},l)$-partition of a graph $G$} if $G[X_i] \in \mc{F}$ for every $i \in [l]$. Let $\mc{P}(\mc{F},l)$ denote the family of all graphs admitting an $(\mc{F},l)$-partition. Clearly, if $\mc{F}$ is hereditary, then so is $\mc{P}(\mc{F},l)$. More generally, we say that 
a collection  $\mc{X}=(X_1,\ldots, X_l)$ of disjoint subsets of $V(G)$ is an \emph{ $(\mc{F},l,m)$-approximation of $G$}  if $\mc{X}$ is an $(\mc{F},l)$-partition of $G[X]$, where $X = \cup_{i \in [l]}X_i$, and $|V(G)-X| \leq m$. 

Let $\mf{F}=(\mc{F}_1,\ldots,\mc{F}_l)$  be a finite sequence of hereditary families. Generalizing the definitions in the previous paragraph, we define an \emph{$\mf{F}$-partition} of a graph $G$ to be a partition $\mc{X}=(X_1,\ldots, X_l)$  of $V(G)$ such that $G[X_i] \in \mc{F}_i$ for every $i \in [l]$. We denote by $\mc{P}(\mf{F})$ or $\mc{P}(\mc{F}_1,\ldots,\mc{F}_l)$ the family of all graphs admitting an $\mf{F}$-partition.

\vskip 10pt

For a graph $H$, let $\iota(H)$ denote the hereditary family of all graphs isomorphic to an induced subgraph of $H$. Given a graph family $\mc{F}$,  let $\mc{F}^{+}$ denote the family of all graphs $G$ such that $G \setminus v \in \mc{F}$ for some $v \in V(G)$. Let $\mc{S}$ denote the family of all edgeless graphs, and let $\mc{C}$ denote the family of all complete graphs. Thus, for example, $\mc{H}(l)=\mc{P}(\mc{S},l)$.

Finally, we say that a property $\mc{P}$ holds for \emph{almost all graphs} in $\mc{F}$, if $$\lim_{n \to \infty}\frac{|\mc{F}^n \cap \mc{P}|}{|\mc{F}^n|} = 1.$$ 

\subsection{Alon-Balogh-Bollob\'as-Morris structure theorem}\label{s:ABBM}

In this subsection we present several results from~\cite{ABBM11}. 

The first one follows immediately from the definition of the coloring number of a hereditary family.

\begin{lem}[{\cite[Theorem 1]{ABBM11}}]\label{l:lower} Let $\mc{F}$ be a  hereditary family, and let $l =\wpn(\mc{F})$ then $$\h(\mc{F},n) \geq \left(1-\frac 1 l \right) \frac{n^2}{2} - O(1).$$
\end{lem} 

Alon, Balogh, Bollob\'as and Morris~\cite{ABBM11} give the following useful characterization of thin graph families, which in particular implies that \cref{t:AlonSpeed} holds for them. 

\begin{thm}[{\cite[Theorem 2, Lemma 7]{ABBM11}}]\label{l:Jfree} Let $\mc{F}$ be a hereditary graph family then the following are equivalent. 
	\begin{itemize}
		\item $\mc{F}$ is thin,
		\item there exists a bigraph $J$ such that every graph in $\mc{F}$ is $J$-free,
		\item there exists $\eps > 0$ such that $ \h(n,\mc{F}) \leq n^{2-\eps}$ for all $n$.
	\end{itemize}	
\end{thm}	

The next theorem provides a structural description of typical graphs in hereditary families, mentioned in the introduction. 

\begin{thm}[{\cite[Theorem 1]{ABBM11}}]\label{t:Alon}
	Let $\mc{F}$ be a hereditary family. Then there exists $\eps >0$ and a thin hereditary family $\mc{T}$ such that almost every  $G \in \mc{F}$ admits an $(\mc{T},\chi_c(\mc{F}),|V(G)|^{1 -\eps})$-approximation. 
\end{thm}

Note that  Theorems~\ref{l:Jfree} and~\ref{t:Alon} imply  \cref{t:AlonSpeed}. 

\vskip 10pt
One can consider refining \cref{t:Alon} in several ways. First, we could attempt replacing an approximation of $G$ by a partition, or reducing the ``error'' of the approximation as much as possible. Second, we could attempt refining the structure so that it provides a certificate of membership in $\mc{F}$.  

In the next section we present a refinement of \cref{t:Alon}, which, under a number of technical conditions on $\mc{F}$, has the first of the properties mentioned above and takes a step towards  the second one. A different refinement of \cref{t:Alon} was established by Reed and Scott~\cite{RS18}, and the ideas we learned from their proof are indispensable for our argument. 

One of the key definitions used in the statement of our refinement of \cref{t:Alon} is the following.
Let $\mc{F}$ be a hereditary graph family, and let $l =\wpn(\mc{F})$. We say that a graph $H$ is \emph{$\mc{F}$-reduced} if there exists an integer $0 \leq s \leq l-1$ such that 
$\mc{P}(\iota(H),\mc{H}(s,l-1-s)) \subseteq \mc{F}$.\footnote{In other words, a graph $H$ is $\mc{F}$-reduced if for some $0 \leq s \leq l-1$ the family  $\mc{F}$ contains all graphs which admit a vertex partition into $l$ parts, such that the first part induces a subgraph of $H$, $s$ of the remaining parts are stable sets, and the rest are cliques.} We say that $H$  is  \emph{$\mc{F}$-dangerous} if it is not $\mc{F}$-reduced. Let $\red(\mc{F})$ and  $\dang(\mc{F})$ denote the families of all $\mc{F}$-reduced and $\mc{F}$-dangerous  graphs, respectively. Note that $\red(\mc{F})$ is a thin hereditary family. 

The following structural description of finitely generated monotone graph families by Balogh, Bollob\'{a}s and Simonovits~\cite{BBS09}, serves as a model for our result.

\begin{thm}[Balogh, Bollob\'{a}s and Simonovits~\cite{BBS09}]\label{BBS}
	For every finitely generated monotone graph family $\mc{F}$ there exists $c>0$ such that almost every graph in $\mc{F}$ admits a $(\red(\mc{F}),\wpn(\mc{F}),c)$-approximation.
\end{thm}	

Note that any monotone family $\mc{F}$ contains the family of all graphs which admit a vertex partition into a set inducing a graph in $\red(\mc{F})$ and $\wpn(\mc{F}) -1$ independent sets. Thus Theorem~\ref{BBS}  ``sandwiches" almost  all graphs in $\mc{F}$ between the above family, and the family $\mc{P}(\red(\mc{F}),\wpn(\mc{F}))$, up to deletion of a constant number of vertices. 

\subsection{Typical structure of apex-free hereditary families}\label{s:structure1}
 
In this section we state our first main theorem. First, we need a few more definitions, which are used to describe the restrictions we place on the hereditary family $\mc{F}$ in our refinement of \cref{t:Alon}.

A \emph{substar} is a subgraph of a star, and an \emph{antisubstar} is a complement of a substar. We say that a hereditary family $\mc{F}$ is \emph{apex-free} if $\crit(\mc{F})$ contains a substar and an antisubstar. It turns out that assuming that $\mc{F}$ is apex-free significantly simplifies analysis of its structure.
The following easy observation gives an alternative description of apex-free hereditary families.

\begin{obs}\label{o:apexfree}
	A hereditary family $\mc{F}$ with $l=\chi_c(\mc{F})$ is apex-free if and only if for every $0 \leq s \leq l$ there exists a graph $H \not \in \mc{F}$ such that $H \setminus u \in \mc{H}(s,l-s)$ for some $u \in V(H)$. 
\end{obs} 


We say that a hereditary family $\mc{F}$ with $l=\wpn(\mc{F}) \geq 2$ is \emph{smooth} if for every $\delta>0$ there exists $n_0$ such that \begin{equation}\label{e:smooth0}
\h(\mc{F},n) \geq  \h(\mc{F},n-1) + ((l-1)/l -\delta)n
\end{equation} for all  integers $n \geq n_0$.  As $\h(\mc{F},n) =  \frac{l-1}{2l}n^2 + o(n^2)$ for every  hereditary family with $\chi_c(\mc{F})=l$, we expect ``reasonable" hereditary families to be smooth, yet we found it difficult to prove that a given  hereditary family is smooth without first understanding its structure.

We are finally ready to state our first main structural result, which is proved in \cref{s:structureproof}.

\begin{restatable}{thm}{critical}\label{t:critical}
	Let $\mc{F}$ be an apex-free hereditary family, let  $l=\wpn(\mc{F}) \geq 2$, let $\mc{K} \subseteq \crit(\mc{F})$ be a finite set of graphs, and let $\mc{T} = \Forb(\mc{K})$. If $\mc{P}(\mc{T},l) \cap \mc{F}$ is smooth then almost all graphs  $G \in \mc{F}$  admit a  $(\mc{T},l)$-partition. 
\end{restatable}

It is likely that the condition  of \cref{t:critical} that $\mc{F}$  is apex-free can be relaxed at the expense of replacing  $(\mc{T},l)$-partition in the conclusion by a $(\mc{T},l,c)$-approximation for some $c = c(\mc{F})$. 
Proving such an extension, however, involves several additional technical difficulties, and, as \cref{t:critical} suffices for our intended applications we chose not to pursue such a generalization.  

Note that if $\mc{F}$ is finitely generated  then there exists a finite $\mc{K} \subseteq \crit(\mc{F})$ such that $\Forb(\mc{K}) = \red(\mc{F})$. Thus \cref{t:critical} implies the following analogue of \cref{BBS} for apex-free hereditary families.

\begin{cor}\label{c:critical} Let $\mc{F}=\Forb(\mc{H})$ be a finitely generated  apex-free hereditary family. If $\mc{P}(\red(\mc{F}),\wpn(\mc{F})) \cap \mc{F}$ is smooth then almost all graphs in $\mc{F}$ admit a  $(\red(\mc{F}),\wpn(\mc{F}))$-partition. 
\end{cor}

To demonstrate the applicability of \cref{c:critical} and demystify some of its conditions we quickly derive from it the following classical theorem of Kolaitis, Pr\"{o}mel and Rotschild~\cite{KPR87}. 

\begin{thm}[Kolaitis, Pr\"{o}mel and Rotschild~\cite{KPR87}]\label{c:KPR}
Almost all $K_{l+1}$-free graphs are $l$-colorable.  
\end{thm}

\begin{proof}
Let $\mc{F}=\Forb(K_{l+1})$. It is easy to check that  $\wpn(\mc{F})=l$, and $\mc{S}=\red(\mc{F})$ is the family of edgeless graphs. The second observation, implies that  $\mc{F}$ is apex-free.  We have $\mc{P}(\mc{S},l)= \mc{H}(l) \subseteq \mc{F}$. (Recall that $ \mc{H}(l)$ is the family of all  $l$-colorable graphs.)   Finally, it is easy to see, and directly follows from \cref{l:cleanext} below, that  $\mc{H}(l)$ is smooth. Hence \cref{c:critical} implies that almost all graphs in $\mc{F}$ are in $\mc{H}(l)$.
\end{proof}

The requirement that $\mc{P}(\mc{T},l) \cap \mc{F}$ is smooth in \cref{t:critical} is particularly technical. Let us present a simple lemma, which can be used to ensure that it is satisfied.  Given $G \in \mc{F}$ with $V(G)=[n]$, let $\ext(G, \mc{F})$ denote the set of all graphs $G' \in \mc{F}$ such that $V(G') = [n+1]$ and $G \subseteq G'$.   
We say that a family $\mc{F}$ is \emph{extendable} if $\chi_c(\mc{F}) \geq 1$, and there exists $n_0 \geq 0$ such that  $|\ext(G, \mc{F})| \geq 1$ for every $G \in \mc{F}$ with $|V(G)|\geq n_0$.
  
  \begin{lem}\label{l:cleanext}
  	Let $\mc{T}_1,\mc{T}_2,\ldots,\mc{T}_l$ be extendable thin graph families. Then  the family $\mc{P}(\mc{T}_1,\mc{T}_2,\ldots,\mc{T}_l)$ is smooth.
  \end{lem}
  
  \begin{proof}
  	Let $\mf{T}=(\mc{T}_1,\mc{T}_2,\ldots,\mc{T}_l)$ and let $\mc{F} = \mc{P}(\mf{T}) = \mc{P}(\mc{T}_1,\mc{T}_2,\ldots,\mc{T}_l)$. Clearly, $\chi_c(\mc{F})=l.$
  	
  Let $n_0$ be such that  $|\ext(G, \mc{T}_i)| \geq 1$ for every $i \in [l]$ and every $G \in \mc{T}_i$ such that $|V(G)| \geq n_0$. An easy counting argument shows that for almost every $G \in \mc{F}$ and every $(\mc{T}_1,\mc{T}_2,\ldots,\mc{T}_l)$-partition $\mc{X}$ of $V(G)$ we have $|X| \geq n_0$ for every $X \in \mc{X}$. 
  
Let  $G \in \mc{F}$ be a graph which has the above property with $V(G)=[n]$, and let $\mc{X}=(X_1,X_2,\ldots,X_l)$ be an $\mf{T}$-partition $\mc{X}$ of $V(G)$. Suppose without loss of generality that
 $|X_1| \leq n/l$. By our assumption $|X_1| \geq n_0$, and so there exists 
  	a graph $G'_1 \in \mc{T}_1$ such that $V(G'_1)=X_1 \cup \{n+1\}$ and $G[X_1] \subseteq G'_1$. We can obtain a graph  in $\ext(G, \mc{F})$ by adding to $G \cup G'_1$  any subset of edges joining $[n+1]$ to $[n] - X_1$, implying that $|\ext(G, \mc{F})| \geq 2^{(l-1)n/l}$.
  	
 It follows that $\h(\mc{F},n+1) \geq \h(\mc{F},n) + (l-1)n/l - O(1)$, implying (\ref{e:smooth0}) for all $n$ sufficiently  large with respect to $\delta$.
  \end{proof}	
  
 We finish this section with a statement of a couple of auxiliary result which are needed to derive \cref{t:Speed} from \cref{t:critical}.  
  
 The applications of our structure theorem use not only existence of a structured partition of a typical graph, but the fact that such a partition is unique and essentially balanced in the following sense. Given $\eps>0$, 
 we say that a partition $\mc{X}$ of an $n$ element set is \emph{$\eps$-balanced} if $|X - n/|\mc{X}|| \leq n^{1 - \eps}$ for all $X \in \mc{X}$. 
 
 We say that a  hereditary family $\mc{F}$ is \emph{meager} if it is thin and apex-free, i.e. $\mc{F}$ is meager if and only if it does not contain some substar and some antisubstar. \footnote{Thus a family $\mc{F}$ is apex-free if and only if $\red(\mc{F})$ is meager.}
 Meager  hereditary families have a significantly more restrictive structure than general thin families. In particular, 
 a rough structure theorem for all (rather than almost all) graphs in a meager hereditary family is given in~\cite{CNRS15}.
 
 \begin{restatable}{lem}{uniqone}\label{l:uniqueness1} Let $\mc{F}$ be a hereditary family with $l = \wpn(\mc{F})$, and let $\mc{T}$ be meager. Then there exists $\eps>0$ such that for almost  all graphs in $G \in \mc{F}$ the following holds. If $G$ admits a $(\mc{T},l)$-partition then such a partition is $\eps$-balanced, and, if additionally  $\mc{F}$ is smooth, then this partition is  unique.   	
 \end{restatable} 
 
 Lemmas~\ref{l:cleanext} and~\ref{l:uniqueness1} have the following immediate useful corollary.
 
 \begin{cor}\label{c:uniqueness} Let $\mf{T}=(\mc{T}_1,\mc{T}_2,\ldots,\mc{T}_l)$, where $\mc{T}_i$ is a meager extendable hereditary family for every $i \in [l]$. Then there exists $\eps>0$ such that  almost  all graphs in $\mc{P}(\mf{T})$ admit a unique  $\mf{T}$-partition and such a partition is $\eps$-balanced.   		
 \end{cor} 
 
 \begin{proof}
Let $\mc{F} = \mc{P}(\mf{T})$ and let  $\mc{U} = \cup_{i=1}^{l}\mc{T}_i$.  Then $\mc{F}$ is smooth by \cref{l:cleanext}, and $\mc{U}$ is a  meager hereditary family. By definition, every graph in $\mc{F}$ admits a $\mf{T}$-partition which is in particular a  $(\mc{U},l)$-partition. Thus by \cref{l:uniqueness1} for almost all graphs in $\mc{F}$ such an $\mf{T}$-partition  is unique and $\eps$-balanced. 
 \end{proof}
  
 Finally, for certain applications we need additional stability properties of structured partitions guaranteed by the next lemma. Let $(G_1,\ldots,G_l)$ be a collection of vertex disjoint graphs. We say that a graph $G$ is an extension of $(G_1,\ldots,G_l)$ if $V(G)=\cup_{i \in [l]}V(G_i)$, and $G_i$ is an induced subgraph of $G$ for every $i \in [l]$.

 \begin{restatable}{lem}{uniqtwo}\label{l:uniqueness2} Let $\mc{T}$ be a meager hereditary family, let $l$ be an integer, and let $\eps>0$ be real. Let $(G_1,\ldots,G_l)$ be a collection of vertex disjoint graphs such that $G_i \in \mc{T}$ for every $i \in [l]$, and $ \mc{X} =(V(G_1),\ldots,V(G_l))$ 
is  an $\eps$-balanced partition of $[n]$. Then $\mc{X}$ is the unique $(\mc{T},l)$-partition of $G$ for almost every extension $G$ of $(G_1,\ldots,G_l)$. 
 \end{restatable}   
  
 Lemmas \ref{l:uniqueness1} and \ref{l:uniqueness2} are proved in \cref{s:structureproof}.
 
\subsection{Exact structure of critical hereditary families}\label{s:critical}

In this section we state our second main structural result, which implies \cref{t:Speed}. It extends a  characterization of hereditary families with subexponential speed due to Alekseev~\cite{Alekseev97} and Balogh, Bollob\'as and Weinreich~\cite{BBW00}, which we now present.

Informally, adjacencies in graphs in such families are determined by a constant number of vertices. 
To formalize this statement, we introduce the following definitions.
We say that a set $X \subseteq V(G)$ is a \emph{crown} of a graph $G$ if for every $v \in V(G)$ either $v$ is adjacent to every vertex of $X$, or  $v$ is not adjacent to any vertex in $X$.\footnote{Note that this implies, in particular, that $G[X]$ is edgeless or complete.}
We say that a set $S \subseteq V(G)$ is a \emph{core} of  $G$ if $V(G)-S$  is a crown of $G$. We say that a graph $G$ is an \emph{$s$-star} for an integer $s \geq 0$ if $G$ has a core of size at most $s$. Thus $0$-stars are exactly complete or edgeless graphs, and $1$-stars are induced subgraphs of stars or antistars. 

We say that a hereditary family $\mc{F}$  is \emph{$s$-star-like} if there exists an integer $n_0$ such that every graph $G \in \mc{F}$ with $|V(G)| \geq n_0$ is an $s$-star. We say that $\mc{F}$ is \emph{star-like} if it is $s$-star-like for some non-negative integer $s$, and that $\mc{F}$ is \emph{non-star-like}, otherwise. The results of~\cite{Alekseev97,BBW00} imply, in particular, that a hereditary family has a subexponential speed if and only if it is star-like. 

Balogh, Bollob\'as and Weinreich~\cite{BBW00} give a detailed structural description of star-like hereditary families. To state it we need  the following definitions. 

Every graph $s$-star $G$ with the core $S$ can be encoded by $G[S]$, the map $\alpha: S \to \{0,1\}$ which determines for every $v \in S$ whether $v $ is adjacent to all vertices in the crown $V(G)-S$ or not, and  $\beta \in \{0,1\}$ which determines whether $G[V(G)-S]$ is complete or edgeless. Based on this way of encoding $s$-stars, we define 
an \emph{$s$-star system} (or simply a \emph{star system}) $\mc{J}$ as a triple $\mc{J}=(J,\alpha, \beta)$, where $J$ is a graph with $|V(J)|\leq s$, $\alpha: V(J) \to \{0,1\}$, and $\beta \in \{0,1\}$.  We say that a star system $(J,\alpha, \beta)$ is \emph{irreducible} if it corresponds to a minimal core, i.e. for every $v \in V(J)$ such that $\alpha(v) = \beta$ there exists $u \in V(J) - \{v\}$ such that either $\alpha(u)=1$ and $uv \not \in E(J)$, or  $\alpha(u)=0$ and $uv \not \in E(J)$.

Given an $s$-star system $\mc{J}=(J,\alpha, \beta)$, define a \emph{$\mc{J}$-template} in a graph $G$ to be an embedding $\psi$ of $J$ into $G$ such that,  denoting the image of $\psi$ by $Z$, we have
\begin{itemize}	
	\item if $\alpha(v)=1$ then $\psi(v)$ is adjacent to every vertex in $V(G) - Z$  in $G$, and, otherwise,
	$\psi(v)$ is adjacent to no vertex in $V(G)- Z$, and,
	\item if $\beta=1$ then $V(G) - Z$  is a clique in $G$, and, otherwise,  $V(G) - Z$  is an independent set.
\end{itemize}

Note, that the above properties imply, in particular, that $Z$ is a core of $G$, and thus $G$ is an $s$-star. Moreover, if $|V(G)-Z|\geq 2$, then  $Z$ is a minimal core of $G$ if and only if $\mc{J}$ is irreducible.
Let $\mc{P}(\mc{J})$ denote the family of induced subgraphs of all graphs which admit a 
$\mc{J}$-template. The following observation is straightforward. 

\begin{obs}\label{o:starsystems}
	Let $G$ be a graph and $s \geq 0$ be an integer. Then $G$ is an $s$-star if and only if $G \in \mc{P}(\mc{J})$ for some  irreducible $s$-star system $\mc{J}$.
\end{obs}
 
 The speed of $\mc{P}(\mc{J})$ was determined in~\cite{BBW00}. 
 
\begin{lem}[\cite{BBW00}]\label{l:starspeed}
 Let	$\mc{J}=(J,\alpha, \beta)$ be an irreducible star system. Then
 $\h(\mc{P}(\mc{J}),n)=|V(J)|\log{n}+O(1).$
\end{lem}

Finally, to present the characterization of minimal non-star-like hereditary families from~\cite{Alekseev97} we need the following additional notation.  Given two families of graphs $\mc{F}_1$ and $\mc{F}_2$, let $\mc{F}_1 \vee \mc{F}_2$ denote the family of all graphs  which are disjoint unions of a graph in $\mc{F}_1$ and a graph in 
$\mc{F}_2$. Similarly, let $\mc{F}_1 \wedge \mc{F}_2$ denote the family of joins of graphs in $\mc{F}_1$ with graphs in $\mc{F}_2$.\footnote{For example, $\mc{S} \vee \mc{S} = \mc{S}$ and $\mc{S} \wedge \mc{S}$ is the family of complete bipartite graphs.} Clearly, if $\mc{F}_1$ and $\mc{F}_2$ are hereditary then so are $\mc{F}_1 \vee \mc{F}_2$  and $\mc{F}_1 \wedge \mc{F}_2$.
For a graph family $\mc{F}$, let  $\bar{\mc{F}}$ denote the family of complements of graphs in $\mc{F}$. Let $\mc{M}$ denote the hereditary family of all matchings, i.e. all graphs with maximum component size two.

The next theorem gives the aforementioned characterization of hereditary families with subexponential speed, which together with \cref{l:starspeed} implies \cref{t:SZSpeed}.

\begin{thm}[\cite{Alekseev97,BBW00}]\label{t:BBW}
	For a hereditary family $\mc{F}$ the following are equivalent
	\begin{description}
		\item[(i)] $\mc{F}$ is star-like,
		\item[(ii)] $h(\mc{F},n)=o(n)$,
		\item[(iii)] neither $\mc{F}$, nor $\bar{\mc{F}}$  contains any of the following families: $\mc{C} \vee \mc{C}, \mc{C} \vee \mc{S}, \mc{M}$ and $\mc{C}^{+}$.
	\end{description}
	Moreover, $\mc{F}$  is $s$-star-like if and only if there exist irreducible $s$-star systems $\mc{J}_1,\ldots,\mc{J}_r$ and an integer $n_0$ such that
	$$\mc{F}^n = \cup_{i=1}^r \mc{P}^n(\mc{J}_i)$$ for every $n \geq n_0$.
\end{thm}

The main application of \cref{t:critical} in this paper is an 
extension of \cref{t:BBW}.  We start by presenting the necessary generalizations of definitions used in the statement of \cref{t:BBW}.
 
First, we extend the notion of a star-like family. 
We say that a family $\mc{F}$ is \emph{$(l,s)$-critical} if $\chi_c(\mc{F})=l$ and $\red(\mc{F})$ is $s$-star-like, i.e. every sufficiently large $\mc{F}$-reduced graph is an $s$-star. We say that a family is  $\mc{F}$ is \emph{critical} if it is $(l,s)$-critical for some $l,s$, and  $\mc{F}$ is \emph{non-critical}, otherwise.

Next we need to generalize the definition of star-systems.
We define an \emph{$(l,s)$-constellation} (or simply a \emph{constellation}) to be a quadruple $\mc{J}=(J,\phi,\alpha, \beta)$, where $J$ is a (possibly empty) graph, and $\phi,\alpha, \beta$ are functions such that $\phi: V(J) \to [l]$ satisfies $|\phi^{-1}(i)| \leq s$ for every $i \in [l]$,  $\alpha : V(J) \to \{0,1\}$ and $ \beta:[l] \to \{0,1\}$. For $i \in [l]$ define  $\mc{J}_i=(J[\phi^{-1}(i)],\alpha|_{\phi^{-1}(i)},\beta(i))$. Thus $\mc{J}_i$ is an $s$-star system.

Given an $(l,s)$-constellation $\mc{J}=(J,\phi,\alpha, \beta)$, define a \emph{$\mc{J}$-template} in a graph $G$ to be a tuple $(\psi,X_1,X_2, \ldots,X_l)$ such that  
\begin{itemize}
	\item $(X_1,X_2, \ldots,X_l)$ is a partition of $V(G)$,
	\item $\psi:V(J) \to V(G)$ is an embedding satisfying $\psi(v) \in X_{\phi(v)}$ for every $v \in V(J)$,
\end{itemize}
and, denoting the image of $\psi$ by $Z$, we have
\begin{itemize}	
	\item if $\alpha(v)=1$ then $\psi(v)$ is adjacent to every vertex in $X_{\phi(v)} - Z$  in $G$, and, otherwise,
	$\psi(v)$ is adjacent to no vertex in $X_{\phi(v)} - Z$, and,
	\item if $\beta(i)=1$ then $X_i - Z$ is a clique in $G$, and, otherwise,  $X_i -Z$ is an independent set.
\end{itemize}
Thus, in particular,  $\psi|_{\phi^{-1}(i)}$ is a $\mc{J}_i$-template in $G[X_i]$, implying that  $G[X_i]$ is an $s$-star. It follows $(1,s)$-constellations are in bijection with $s$-star systems, and this bijection carries over to templates. We say that a constellation $\mc{J}$ is  \emph{irreducible} if $\mc{J}_i$ is irreducible for every $i \in [l]$. 
As before, let $\mc{P}(\mc{J})$ denote the family of induced subgraphs of all graphs which admit a $\mc{J}$-template. 

We are now ready to state our extensions of  \cref{l:starspeed}  and \cref{t:BBW}.

\begin{restatable}{lem}{constellationspeed}\label{l:constellationspeed} Let
$\mc{J}=(J,\phi,\alpha, \beta)$ be an irreducible $(l,s)$-constellation. Then
	$\h(\mc{P}(\mc{J}),n)-\h(\mc{H}(l),n)=|V(J)|\log{n}+O(1).$
\end{restatable}

\begin{restatable}{thm}{criticaltwo}\label{t:critical2} Let $\mc{F}$ be a hereditary family, and let $l = \chi_c(\mc{F})$. Then  the following are equivalent
	\begin{description}
		\item[(i)] $\mc{F}$ is critical,
		\item[(ii)] $h(\mc{F},n)-h(\mc{H}(l),n)=o(n)$,
		\item[(iii)] for any choice of hereditary families $\mc{F}_1,\mc{F}_2,\ldots,\mc{F}_l$ such that \begin{itemize} \item either $\mc{F}_1$ or $\bar{\mc{F}}_1$ is one of the families $\mc{C} \vee \mc{C}, \mc{C} \vee \mc{S}, \mc{M}$, $\mc{C}^{+}$,  and \item  $\mc{F}_i \in \{\mc{C},\mc{S}\}$ for every $2 \leq i \leq l$, \end{itemize} we have $$\mc{P}(\mc{F}_1,\ldots,\mc{F}_l) \not \subseteq  \mc{F}.$$
	\end{description}
Moreover, if $\mc{F}$ is $(l,s)$-critical then there exist irreducible $(l,s)$-constellations $\mc{J}_1,\ldots,\mc{J}_r$ such that $\cup_{i=1}^r \mc{P}^n(\mc{J}_i) \subseteq \mc{F}$, and almost every graph in $\mc{F}$ lies in $\cup_{i=1}^r \mc{P}^n(\mc{J}_i)$.
\end{restatable}

Note that \cref{l:starspeed}  and \cref{t:critical2} imply  \cref{t:Speed}, as promised, and that  \cref{t:BBW} corresponds to the case $l=1$ of \cref{t:critical2}, except that the last conclusion of  \cref{t:BBW} is weakened. We prove \cref{l:starspeed} and \cref{t:critical2} in \cref{s:applications}.

\vskip 10pt

We finish this section by comparing a case $s=0$ of \cref{t:critical2} with the results of \cite{BB11}. Note that for every $(l,0)$-constellation $\mc{J}$ there exists an integer $0 \leq s \leq l$ such that $\mc{P}(\mc{J}) = \mc{H}(s,l-s)$. Thus  \cref{t:critical} in this case implies the following.

\begin{cor}\label{c:BB11} Let  $\mc{F}$ be an $(l,0)$-critical hereditary family. Then for almost all $G \in \mc{F}$ we have $G  \in \mc{H}(s,l-s)$ for some $0 \leq s \leq l$ such that  $\mc{H}(s,l-s) \subseteq \mc{F}$.
\end{cor} 

In~\cite{BB11} Balogh and Butterfield, define a graph $H$ to be \emph{critical} if  every sufficiently large $\Forb(H)$-reduced graph is complete or edgeless. Thus $H$ is critical if and only if $\Forb(H)$ is $(l,0)$-critical, where $l = \chi_c(\Forb(H))$. They prove the following.
   
\begin{thm}[\cite{BB11}]\label{t:BB11} Let $H$ be  a critical graph, and let $l=\chi_c(\Forb(H))$. Then for almost all $G \in \Forb(H)$ we have $G  \in \mc{H}(s,l-s)$ for some $0 \leq s \leq l$ such that  $H \not \in \mc{H}(s,l-s)$.
\end{thm} 

Clearly, \cref{t:BB11} follows from \cref{c:BB11},\footnote{We suspect that  the proof of \cref{t:BB11} in \cite{BB11} can be easily used to establish \cref{c:BB11}}, as we alluded in the abstract.
  
\section{Typical structure  via Kolmogorov complexity: Proof of \cref{t:critical}.}\label{s:structureproof}

This section contains the proof of \cref{t:critical} and Lemmas~\ref{l:uniqueness1} and~\ref{l:uniqueness2}. As mentioned in the introduction we present our counting arguments in the language of  Kolmogorov complexity. 

\subsection{Brief motivation}

First, let us briefly motivate our choice of Kolmogorov complexity as the language, which is unusual in the area.

Let $\mc{F}$ be a hereditary family the typical structure of which we'd like to establish. That is our goal is to show that almost all graphs in $\mc{F}$ belong to some explicitly given ``structured" $\mc{F}_*$ subfamily of $\mc{F}$, or equivalently \begin{equation}
\label{e:almostall}h(\mc{F},n)-h(\mc{F}-\mc{F}_*,n) \underset{n \to \infty}{\rightarrow} \infty.
\end{equation}
Establishing (\ref{e:almostall}) typically involves proving upper bounds on $h(\mc{F}-\mc{F}_*,n)$ by exhibiting a compact encoding of the graphs in $(\mc{F}-\mc{F}_*)^n$. 

Here is a simple example of an argument of this type. Let $\mc{F}$ is a smooth hereditary family with $\chi_c(\mc{F}) \geq 2$. We claim that almost all graphs in $G \in \mc{F}$ have minimum degree at least $(1-\eps)|V(G)|/4$ for every fixed $\eps > 0$. Indeed, let $\mc{F}'$ be the family of graphs in $ \mc{F}$ which don't have this property. Then every graph $G \in (\mc{F}')^n$ with a vertex $i$ of degree less than $(1-\eps)n/4$ can be encoded as a triple of $(G \setminus i, i, N_G(i))$, where $N_G(i)$ denotes the set of neighbors of $i$. As $G \setminus i$ can be considered as a graph in $\mc{F}^{n-1}$  the number of such encodings gives a bound on there are at most \begin{align*} 2&^{h(\mc{F},n-1)}\cdot n \cdot\left(\sum_{0 \leq j < (1-\eps)n/4}\binom{n}{j}\right) \\ &= 2^{h(n-1)+(-(1-\eps)/4\cdot\log((1-\eps)/4) -(3+\eps)/4\cdot\log((3+\eps)/4))n+o(n)} \\&= 2^{h(\mc{F},n)-\Omega(n)},\end{align*}      
where the first equality uses standard approximation of the binomial coefficient, and the second equality holds as $\mc{F}$ smooth, demonstrating a typical way we exploit the smoothness of the families under consideration. It follows that $h(\mc{F}',n)=h(\mc{F},n)-\Omega(n)$, establishing (\ref{e:almostall}) in our example. 

In our proofs, we repeatedly use  encoding arguments as above. To allow application of the intermediate results without keeping track of the nature of encodings we use, we keep track of a single parameter: the length of the shortest  encoding a given graph. Kolmogorov complexity is the formal way of defining  such length and thus is extremely convenient for our purposes.

\subsection{Basics of Kolmogorov complexity}    

In this subsection we  introduce the necessary terminology and giving a quick overview of the properties of Kolmogorov complexity. 

Let $\mf{A}$ be a finite alphabet such that $\{0,1\} \subseteq \mf{A}$. It will be additionally convenient to assume that the comma symbol ``," is an element of $\mf{A}$. Let $\mf{A}^*$ denote the set of finite sequences of elements from $\mf{A}$, called \emph{strings}, and we refer  the set of finite sequences of zeros and ones as  \emph{binary strings}.

Fix an arbitrary universal Turing machine $U$ which accepts strings in $\mf{A}^*$ as inputs. The \emph{Kolmogorov complexity $C(X)$} of a string $X \in \mf{A}^* $ is defined as the minimum length of a binary string $B$ such that $U(B) = X$,  i. e. $U$ outputs $X$ given $B$ as an input. Intuitively, $C(X)$ is the length of the shortest description of $X$, and this intuition is largely sufficient for our purposes.  Given $X, Y \in \mf{A}^*$ the \emph{conditional complexity of $X$ given $Y$}, denote by $C(X|Y)$ is the minimal length of a binary string $B$ such that $U(B,Y)=X$.\footnote{Here $B,Y$ denotes a string in  $\mf{A}^*$ obtained by concatenating $B$, the comma symbol, and $Y$.} 

We will need the following basic and hopefully natural properties of Kolmogorov complexity, proved e.g. in~\cite{LiVi08}. We also refer the reader to~\cite{LiVi08} for more detailed discussion of the concept.
 
\begin{obs}\label{o:Kolmogorov1}
	\begin{description}
		\item[(C0)] $C(X) \leq |X|+O(1)$ for every binary string $X$.
		\item[(C1)] $C(X,Y)\leq C(X)+C(Y) + \log C(Y) + o(\log C(Y))$ for all strings $X$ and $Y$. In particular, $C(X) \leq C(X|Y)+C(Y)+ \log C(Y)  + o(\log C(Y))$.
		\item[(C2)] For every computable function $f$ there exists a constant $c_f$ such that $C(f(X))\leq C(X)+c_f$.	
	\end{description}
\end{obs}

We will be considering Kolmogorov complexity of graphs and bigraphs, and to do this we assume for the duration
of this section that all the graphs and bigraphs are ordered. This allows us when needed to identify the vertex set of an $n$ vertex graph with $[n]$.\footnote{We do, however, need to be somewhat careful when we do the identification, as this identification does not commute with the operation of taking subgraphs.} Thus any graph on $n$ vertices can be represented by a binary string of length $\binom{n}{2}$, and a $(n_1,n_2)$-bigraph can be represented by a string of length $n_1n_2$.  

The following additional useful properties follow directly from \cref{o:Kolmogorov1}.

\begin{cor}\label{c:Kolmogorov2}
	\begin{description}
	\item [(C3)] Let $X$ be a binary string of length $n$ with at most $m$ non-zero elements. Then $C(X) \leq m(\log n - \log m)+ 2m + O(\log n)$.
	\item [(C4)] Let $\mc{P}$ be a computable graph property. Then $$C(G) \leq \h(\mc{P},|V(G)|)+O(\log |V(G)|)$$ for every $G \in \mc{P}$.
	\item [(C5)] Conversely, for every graph property $\mc{P}$ and every $f: \bb{N} \to \bb{R}$ such that $\lim_{n \to \infty} f(n) = + \infty$ we have  $$C(G) \geq \h(\mc{P},|V(G)|)-f(|V(G)|)$$ for almost all $G \in \mc{P}$.
    \end{description}	
\end{cor}

\begin{proof}
	For the proof of (C3), let $S$ be the set of binary strings of length $n$ with at most $k$ non-zero elements. Then $|S| = \sum_{i=0}^k\binom{n}{i} \leq n(en/m)^m$. As there exists a computable function that returns $X \in S$, given $n$ and a number of the string $X$ in the natural enumeration of elements of $S$, the required bound on $C(X)$ follows from \cref{o:Kolmogorov1}.
	
	The property (C4) holds as $G$ can be computed from $|V(G)|$ and the number of $G$ in any enumeration of the graphs in $\mc{P}$ on $n$ vertices.
	
	Finally, (C5) holds as the total number of  strings $X \in \mf{A}^*$ with $C(X) \leq N$ is no greater than $2^N$.
\end{proof}

For an ordered graph $G$ and disjoint $X,Y \subseteq V(G)$, let 
$$C_G(X,Y) = C(G[X,Y]\;|\;G[X],G[Y]).$$ 
That is, $C_G(X,Y)$ is the conditional complexity of determining edges between $X$ and $Y$ given the subgraph of $G$ induced on each of $X$ and $Y$. For brevity we denote $C_G(X,V(G)-X)$ by $C_G(X)$, and write $C_G(v)$ instead of $C_G(\{v\})$ for $v \in V(G)$.

\subsection{First applications}

As a first application of our newly introduced terminology, we will bound the complexity of graphs that admit an approximation guaranteed by Theorem~\ref{t:Alon}. Simultaneously, we wish to establish further properties of such approximation of a typical graph in a hereditary property. First, a simple observation.

\begin{obs}\label{o:partition1}
	Let $G$ be a graph  with $V(G)=[n]$, and let $\mc{P}$ be a partition of $[n]$ into $k \geq 2$ parts. Then
	\begin{equation}\label{e:partition1}
	C(G) \leq \sum_{X \in \mc{P}}C(G[X]) + \sum_{X,Y \in \mc{P}, X \neq Y} C_G(X,Y) + O(n \log k)
	\end{equation}
\end{obs}

\begin{proof} As there are at most $k^n$ partitions of $[n]$ into $k$ parts, we have $C(\mc{P}) =  O(n \log k)$.   Clearly,
	$G$ can be computed from $\mc{P}$, $(G[X])_{X \in \mc{P}}$ and $(G[X,Y])_{X,Y \in \mc{P}}$, and thus (\ref{e:partition1}) follows from \cref{o:Kolmogorov1}.
\end{proof}

Given a bigraph $J$ denote by $\eps(J)$ the minimum real number such that $\h(\Forb(J),n) \leq n^{2 - \eps(J)}$ for every positive integer $n$. \cref{l:Jfree} implies that $\eps(J) >0$ for every bigraph $J$.
For a collection of disjoint subsets $\mc{X}$ of the vertex set of a graph $G$, we denote by $\h(\mc{X})$ the number of pairs of vertices of $G$ which do not belong to any set of $\mc{X}$. That is $\h(\mc{X}) = \binom{|V(G)|}{2} - \sum_{X \in \mc{X}}\binom{X}{2}$.

\begin{cor}\label{c:approx1} Let $J$ be a bigraph,  let 
$G$ be a graph  on $n$ vertices, and let $\mc{X}$ be a $(\Forb(J),l,m)$-approximation  of $G$ with $m =o(n)$. Then 
\begin{equation}\label{e:partition2}
	C(G) \leq \h(\mc{X}) + n^{2-\eps(J)} + m^2 + O(n \log n),
\end{equation}
and, in particular,
\begin{equation}\label{e:partition3}
	C(G) \leq \frac{(l-1)n^2}{2l} + n^{2-\eps(J)} + nm + O(n \log n)
\end{equation}	
\end{cor}
\begin{proof}
	By Corollary~\ref{c:Kolmogorov2} (C4), we have $C(G[X]) \leq |X|^{2-\eps(J)} + O(\log n)$ for every $X \in \mc{X}$.  Thus we have 
	\begin{equation}\label{e:parts}
	\sum_{X \in \mc{X}}C(G[X])  \leq n^{2-\eps(J)} +O(l\log n).
	\end{equation}
	Let $Z=V(G) - \cup_{X \in \mc{X}}X$ , and let $\mc{P}=\mc{X} \cup \{Z\}$ be a partition of $V(G)$. By Observation~\ref{o:Kolmogorov1}  we have $C_G(X,Y) \leq |X||Y| + O(1)$ for all $X,Y \in \mc{P}$, and so
	\begin{equation}\label{e:rest}
		C(G[Z])+ \sum_{X,Y \in \mc{P}, X \neq Y} C_G(X,Y) \leq m^2 + \h(\mc{X}) + O(n \log n)
	\end{equation}
	Substituting (\ref{e:parts}) and (\ref{e:rest}) into (\ref{e:partition1}) yields (\ref{e:partition2}). Finally, note that by convexity we have
	$$\h(\mc{X}) \leq \binom{n}{2} - l\binom{(n-m)/l}{2} \leq \frac{(l-1)n^2}{2l} + \frac{nm}{l},$$
	and therefore (\ref{e:partition2}) implies (\ref{e:partition3}).
\end{proof}

\subsection{Perfect approximations}

In this subsection we further refine \cref{t:critical} by showing that for almost graphs admitting approximation as in \cref{t:critical} the edges between the parts of approximation are in certain sense randomly distributed.

For a finite set $X$ a \emph{$k$-base $\mc{A}$ on $X$} is a collection of  disjoint ordered $k$ element subsets of $X$. We say that a pair $(\mc{A},\mc{B})$ such that $\mc{A}$ is a $k_1$-base  on $V(G)$  and $\mc{B}$ is an $k_2$-base  on $V(G)$ \emph{induces an ordered $(k_1,k_2)$-bigraph $J$ in $G$} if  $(A,B)$ induces $J$ for some $A \in \mc{A},B \in \mc{B}$, and we say that $(\mc{A},\mc{B})$ is \emph{$J$-free}, otherwise. We say that  $(\mc{A},\mc{B})$ is \emph{restricted}  if it is $J$-free for some ordered $(k_1,k_2)$-bigraph  $J$.

To state the next lemma we will need to extend some of the already established notation from graph properties to bigraph properties. 
Given a bigraph property $\mc{P}$, let $\mc{P}^{n_1,n_2}$ denote the set of all bigraphs in $\mc{P}$ of the form  $(H,[n_1],[n_1+1,n_1+n_2])$\footnote{Here $[n_1+1, n_1,+n_2]$ denotes $\{n_1+1,n_1+2.\ldots,n_1+n_2\}$}, and let  $\h(n_1,n_2,\mc{P}) = \log|\mc{P}^{n_1,n_2}|$.

\begin{lem}\label{l:bases}
	Let $G$ be a graph on $n$ vertices, let $k_1,k_2$ be positive integers, let $X,Y \subseteq V(G)$ be disjoint, let $\mc{A}$ be a $k_1$-base on $X$, and let $\mc{B}$ be a $k_2$-base on $Y$.  Let $\mc{P}$ be a computable bigraph property such that $G[A,B] \in \mc{P}$ for all $A \in \mc{A}$, $B \in \mc{B}$. Then
	\begin{align}\label{e:bases1}
	C_G(X,Y) &\leq |X||Y| + C(\mc{A}|G[X]) + C(\mc{B}|G[Y]) \notag \\ &- (k_1k_2 - \h(k_1,k_2,\mc{P}))|\mc{A}||\mc{B}| + O(\log n). 
	\end{align}
	In particular, if $(\mc{A},\mc{B})$ is restricted, then 
	\begin{equation}\label{e:bases2}
	C_G(X,Y) \leq |X||Y| - \frac{|\mc{A}||\mc{B}|}{2^{k_1k_2}} + O(n \log n) 
	\end{equation}
\end{lem}
\begin{proof}
	Let $a = |\mc{A}|$, $b = |\mc{B}|$, $X' = \cup_{A \in \mc{A}} A$, $Y' = \cup_{B \in \mc{B}} B$. We have \begin{equation}\label{e:bases3}
	C(G[X',Y'] | \mc{A}, \mc{B}) \leq ab \cdot\h(k_1,k_2,\mc{P}) + O(\log n),
	\end{equation}
 as there are $|\mc{P}^{k_1,k_2}|^{ab}$ possibilities for the sequence
 	$(G[A,B])_{A \in \mc{A},B\in \mc{B}}$ satisfying  $G[A,B] \in \mc{P}$ for all $A \in \mc{A}$, $B \in \mc{B}$, and each such sequence can be encoded  as a binary string of length $ab \cdot \h(k_1,k_2,\mc{P}) + O(1)$, which is decodable given $k_1,k_2$ and $\mc{P}$.
 	By \cref{o:Kolmogorov1}, we have
 	\begin{align}\label{e:bases4}
 	 C_G(X,Y) \leq &C_G(X-X',Y)+C_G(X',Y-Y')+ C(\mc{A}|G[X]) + C(\mc{B}|G[Y])\notag \\ &+C(G[X',Y'] | \mc{A}, \mc{B}) + O(\log n) \notag \\ 
 	 \leq &|X||Y|- |X'||Y'|  + C(\mc{A}|G[X]) + C(\mc{B}|G[Y])\notag\\ &+C(G[X',Y'] | \mc{A}, \mc{B})+ O(\log n) .  
 	\end{align}
 	Combining (\ref{e:bases3}) and (\ref{e:bases4}) yields (\ref{e:bases1}). 
 	
 	To deduce (\ref{e:bases2}), let $J$ be a $(k_1,k_2)$-bigraph such that $(\mc{A},\mc{B})$ is $J$-free and let $\mc{P}$ be the family of $J$-free bigraphs. Then \begin{equation}\label{e:bases5} \h(k_1,k_2,\mc{P}) = k_1k_2 + \log\left(1 - \frac{1}{2^{k_1k_2}}\right) \leq k_1k_2 -  \frac{1}{2^{k_1k_2}}.
 	\end{equation}
 	Note further that $C(\mc{A})  = O(n \log n)$ for any base $\mc{A}$ on $[n]$. Substituting this bound and (\ref{e:bases5}) into (\ref{e:bases1}) we obtain (\ref{e:bases2}).	
\end{proof}

We say that a subset $X$ of $V(G)$ is \emph{ $(k,s)$-restricted} if there exists a $k$-base $\mc{A}$ on $X$ and a $k$-base $\mc{B}$ on $V(G)-X$ such that $|\mc{A}|,|\mc{B}| \geq s$ and  $(\mc{A},\mc{B})$ is restricted. We say that $X$ is \emph{$(k,s)$-unrestricted}, otherwise. We say that a collection $\mc{X}$ of disjoint subsets of $V(G)$ is \emph{$(k,s)$-unrestricted} if the set $\cup_{X \in \mc{X}'}X$  is $(k,s)$-unrestricted for every $\mc{X}' \subset \mc{X}$ such that $0 \neq \mc{X}' \neq \mc{X}$, and otherwise  we say that $\mc{X}$ is \emph{$(k,s)$-restricted}.

Combining  the bounds we established so far, we bound the complexity of graphs admitting a  $(k,s)$-restricted $(\Forb(J),l,m)$-approximation.

\begin{lem}\label{l:approx2} Let $k,l,m,s$ be positive integers, and let $J$ be a bigraph such that $k \geq |V(J)|$. 
 	Let $G$ be a graph on $n$ vertices admitting a $(k,s)$-restricted $(\Forb(J),l,m)$-approximation. Then  
 	\begin{equation}\label{e:restricted}
 	C(G) \leq \frac{(l-1)n^2}{2l} + n^{2-\eps(J)} + nm - \frac{s^2}{2^{k^2}} + O(n \log n)
 	\end{equation}
\end{lem} 
\begin{proof} 
	Let $\mc{X}$ be a  $J$-free $(k,s)$-restricted $(l,m)$-approximation of $G$, and $\mc{X}' \subset \mc{X}$ such that $0 \neq \mc{X}' \neq \mc{X}$, and $Y=\cup_{X \in \mc{X}'}X$   is $(k,s)$-restricted. Let $|Y|=y$, let $l'=|\mc{X}'|$, and let $Z=V(G)-Y$. Note that  $\mc{X}'$ is a  $J$-free $(l',0)$-approximation of $G[Y]$ and $\mc{X}-\mc{X}'$ is a $J$-free $(l-l',m)$-approximation of $G[Z]$, by Corollary~\ref{c:approx1} we have
	\begin{align}
	C(G[Y]) &\leq \frac{(l'-1)y^2}{2l'} + y^{2-\eps(J)} + O(n \log n), \label{e:partition4}\\
	C(G[Z]) &\leq \frac{(l-l'-1)(n-y)^2}{2(l-l')} + (n-y)^{2-\eps(J)} + nm + O(n \log n). \label{e:partition4+}
	\end{align}
	Meanwhile, Lemma~\ref{l:bases} implies 
	\begin{equation}\label{e:partition5}
		C_G(Y,Z) \leq y(n-y) - \frac{s^2}{2^{k^2}} + O(n \log n). 
	\end{equation}
	 Substituting (\ref{e:partition4}), (\ref{e:partition4+}) and (\ref{e:partition5}) into the inequality (\ref{e:partition1}) applied to the partition $(Y,Z)$ of $V(G)$, we obtain
	\begin{align*}
		C(G) \leq &y(n-y) + \frac{(l'-1)y^2}{2l'} +\frac{(l-l'-1)(n-y)^2}{2(l-l')}   \\ &+ n^{2-\eps(J)} + nm - \frac{s^2}{2^{k^2}} + O(n \log n). 
    \end{align*}
	As  $$y(n-y) + \frac{(l'-1)y^2}{2l'} +\frac{(l-l'-1)(n-y)^2}{2(l-l')} \leq  \frac{(l-1)n^2}{2l}$$ for all $0 \leq y \leq n$ and $1 \leq l' \leq l-1$, (\ref{e:restricted}) follows.
\end{proof}

Let $G$ be a graph with $n=|V(G)|$.
We say that a collection of disjoint subsets $\mc{X}$ of $V(G)$ with $|\mc{X}|=l$ is an \emph{$(l,k,m)$-perfect approximation of $G$} for a triple of integers $l,k,m \geq 0$ if
	\begin{itemize}
		\item There exists a $(k,k)$-bigraph $J$ such that $\mc{X}$ is a $(\Forb(J),l,m)$-approximation of $G$,
		\item $\mc{X}$ is $(k,m)$-unrestricted, and
		\item $||X| - n/l|\leq m$ for every $X \in \mc{X}$.
	\end{itemize}

\begin{lem}\label{l:balanced} For every bigraph $J$, all $0 < \eps <\eps(J)/3$ ,  and positive integers $k,l,C >0$ there exist $n_0$ such that the following holds.
	Let $G$ be a graph with $|V(G)|=n \geq n_0$ and 
	\begin{equation}\label{e:balanced} C(G) \geq \frac{(l-1)}{2l}n^2- Cn^{2-3\eps},\end{equation} 
	and let $\mc{X}$ be a  $(\Forb(J),l,n^{1 - 3\eps})$-approximation of $G$. Then $\mc{X}$ is $(l,k,n^{1 - \eps})$-perfect. 
\end{lem}

\begin{proof}
We assume without loss of generality that $k \geq |V(J)|$.		By Lemma~\ref{l:approx2}, if   $\mc{X}$ is $(k,n^{1 - \eps})$-restricted then $$C(G) \leq \frac{(l-1)n^2}{2l} - \frac{n^{2 - 2\eps}}{2^{k^2}} + n^{2 - 2\eps(J)} + n^{2 - 3\eps} + O(n\log n),$$  in contradiction with (\ref{e:balanced}), given that $n_0$ is chosen large enough as a function of $C$,$k$, $\eps$ and the constant hidden in $O$-notation. Thus  $\mc{X}$ is $(k,n^{1 - \eps})$-unrestricted.

It remains to show that  $||X| - n/l|\leq n^{1 - \eps}$ for every $X \in \mc{X}$. Let $x = |X|$. Applying Corollary~\ref{c:approx1}, as in the proof of ~\cref{l:approx2}, we obtain,
\begin{align*}
C(G) & \leq h(\mc{X}) +n^{2 - \eps(J)} + O(n^{2 - 3\eps}) \\&\leq x(n-x) + \frac{(l-2)(n-x)^2}{2(l-1)} + O(n^{2 - 3\eps}) \\ &= \frac{(l-1)n^2}{2l} - \frac{l}{2(l-1)}\left(x-\frac {n}{l}\right)^2  +  O(n^{2 - 3\eps}). 
\end{align*}
Therefore by(\ref{e:balanced}) we have  $|x - n/l|=O(n^{1- 3\eps/2})$, and so $|x - n/l|\leq n^{1 - \eps}$ for sufficiently large $n$, as desired.
\end{proof}

The following theorem is our first refinement of Theorem~\ref{t:Alon} in route to Theorem~\ref{t:critical}.

\begin{thm}\label{t:approx}
	Let $\mc{F}$ be a hereditary family with $l=\wpn(\mc{F})$. Then there exist $\eps >0$ such that almost all  graphs in $G \in \mc{F}$ admit an $(l,k,|V(G)|^{1 - \eps})$-perfect approximation.
\end{thm}

\begin{proof}	From Lemma~\ref{l:lower} and Corollary~\ref{c:Kolmogorov2} (C5), it follows that $$C(G) \geq \frac{(l-1)}{2l}n^2- O(n)$$ for almost all graphs $G \in \mc{F}$ on $n$ vertices. 
	
	By Theorem~\ref{t:Alon} there exist $\eps>0$ and a bigraph $J$ such that almost all  graphs in $G \in \mc{F}$ admit a  $(\Forb(J),l,|V(G)|^{1 - 3\eps})$-approximation. By decreasing $\eps$ if necessary, we can further assume that $\eps < \eps(J)/3$. 
	
	It follows from  \cref{l:balanced} that if $n$ is large enough then the graphs $G$ satisfying the above two properties admit  an $(l,k,|V(G)|^{1 - \eps})$-perfect approximation, as desired. 
\end{proof}

We take advantage of the fact that parts  of a perfect approximation are unrestricted in several ways. First, we can use it to build larger subgraphs from pieces which are present sufficiently often in the parts.  This is formalized in the next lemma. 

\begin{lem}\label{l:unrestricted} Let $G$ be a graph, let $s,t,k_1,\ldots,k_t$ be positive integers, let $k=\sum_{i=1}^t k_i$, and let $\mc{X}=(X_1,X_2,\ldots,X_t)$ be a $(k,s)$-unrestricted collection of disjoint subsets of $V(G)$. Let $\mc{A}_i$ be a $(k_i,(t-1)s)$-base in $X_i$ for  $1 \leq i\leq t$, and  let $J(i,j)$ be a $(k_i,k_j)$-ordered bigraph for  $1 \leq i <j \leq t$.   Then there exists $A_1 \in \mc{A}_1, A_2 \in \mc{A}_2 \ldots, A_t \in \mc{A}_t$  such that $G[A_i,A_j]$ is isomorphic to $J(i,j)$ for all   $1 \leq i <j \leq t$.
\end{lem}

\begin{proof}
	We prove the lemma by induction on $t$. The base case $t=2$ is trivial.
	
	Thus we assume $t \geq 3$ and, for the induction step,  consider an auxiliary bipartite graph $F$ with bipartition $(\mc{A}_{t-1},\mc{A}_t)$ where $A \in \mc{A}_{t-1}$ and $B \in \mc{A}_{t}$ are joined by an edge if $G[A,B]$ is isomorphic to $J(t-1,t)$. As  $X_{t-1}$ and $X_t$ are  $(k,s)$-unrestricted, every vertex cover of $F$ has size at least $(t-2)s$. Thus $F$ contains a matching $M$ of size $(t-2)s$. Let $k'_{t-1}=k_{t-1}+k_t$, let $X'_{t-1} = X_{t-1} \cup X_t$, and let $\mc{A}'_{t-1}$ consist of unions of pairs of ordered sets which form edges of $M$,  where we order $A \cup B$ for $A \in \mc{A}_{t-1}$ and $B \in \mc{A}_{t}$ by insisting that all elements of $A$ follow the elements of $B$. Thus  $\mc{A}'_{t-1}$ is a $(k'_{t-1},(t-2)s)$-base.
	
	Finally, for every $1 \leq i \leq t-2$ let $J'(i,t-1)$ be obtained from the ordered bigraphs $J(i,t-1)=(H,A,B)$ and  $J(i,t)=(H',A',B')$ by identifying  $A$ and $A'$ and ordering the other side of the bipartition $B \cup B'$ so that $B$ precedes $B'$.  
	
	Let $\mc{X}'=(X_1,X_2,\ldots,X_{t-2},X'_{t-1})$, and let $\mc{A}'_i$ be a subset of $\mc{A}_i$ of size $(t-2)s$ for $1 \leq i \leq t-2$. It is straightforward to verify  that $\mc{X}'$, $\mc{A}'_1,\ldots,\mc{A}'_{t-1}$ satisfy the conditions necessary for the induction hypothesis to be applicable.  Applying the induction hypothesis to bigraphs $J'(i,j)$ for  $1 \leq i <j \leq t-1$, where $J'(i,j)=J(i,j)$ if $j < t-2$, we obtain the desired conclusion in the induction step.
\end{proof}

Second useful property of a perfect approximation is that any two such approximations only differ on a small set vertices. We formalize this observation in the following  lemma.

\begin{lem}\label{l:approxClose}
	Let $G$ be a graph on $n$ vertices, let $J$ be a  $(k,k)$-bigraph, and let $X,Y \subseteq V(G)$ be such that $G[X]$ is $J$-free and $Y$ is $(k,s)$-unrestricted. Then $\min\{|X-Y|,|X \cap Y|\} \leq ks$.  
\end{lem}

\begin{proof}
	Assume to the contrary that $|X-Y|\ge ks$ and $|X\cap Y|\ge ks$. Choose an arbitrary $(k,s)$-base $\mc{A}$ on $X-Y$ and a   $(k,s)$-base $\mc{B}$ on $X \cap Y$. Then $(\mc{A},\mc{B})$ is $J$-free, as $G[X]$ is $J$-free, in contradiction to the assumption that $Y$ is $(k,s)$-unrestricted.
\end{proof}

\cref{l:approxClose} in particular shows that every graph has few perfect approximation, and thus each such approximation has low complexity, conditioned on the graph.

 \begin{cor}\label{c:perfect2}
 	Let $G$ be a graph on $n$ vertices, and let  $\mc{P}$ be an $(l,k,m)$-perfect approximation of $G$ for some integers $k,l,m>0$. Then $$C(\mc{P} \:|\: G ) = O(k l^2 m\log n).$$ 
 \end{cor}
 \begin{proof} There exists a computable function $\mc{P}_*$ that given a graph $G$, and integers $k$ and $l$ returns an  $(l,k,m')$-perfect approximation $\mc{P}_*=\mc{P}_*(G,l,k)$ of $G$ such that $m'$ is as small as possible. That is $C(\mc{P}_*\:|\: G )=O(k+l)$. 
 	
 	As $\mc{P}=(X_1,\ldots,X_l)$ is an $(l,k,m)$-perfect approximation of $G$ then  $\mc{P}_*=(X'_1,\ldots,X'_l)$ is also $(l,k,m)$-perfect. By~\cref{l:approxClose} for each $1 \leq i \leq l$ we either have $|X_i| \leq klm$, or $|X_i \triangle X_j'| \leq 2km$ for some $1 \leq j \leq l$. Thus $C( X_i | \mc{P}_*, G) \leq klm \log n + O(\log n),$  implying the corollary.  
 \end{proof}
 
\subsection{Additional technical lemmas}
 
In our further refinements of \cref{t:approx}  we start  with a perfect approximation as the base of our structure and  analyze the bigraphs $G[Z,X]$ where $Z$ is a small set of vertices, and $X$ is a part of our approximation. On one hand, if $\h(n,\mc{F})$ is quickly increasing, then $G[Z,X]$ must have high complexity for most of the parts $X$. On the other hand, we should be unable to construct a graph not in $\mc{F}$ from $G[Z]$ and subgraphs induced on the parts.
Kolmogorov complexity is particularly convenient as the bookkeeping language for counterbalancing these two requirements. 

Let us start preparing the ingredients necessary to implement this strategy.
First, we need an analogue of \cref{o:partition1} in the bigraph setting, which similarly straightforwardly follows from \cref{o:Kolmogorov1}. 

 \begin{obs}\label{o:partition2} Let $G$ be a graph, let  $X,Y \subseteq V(G)$ be disjoint, and let $ \mc{P}$ be a partition of $Y$.
	 	Then $$C_G(X,Y) \leq C(\mc{P}|G[Y])+\sum_{Y' \in \mc{P}}C(G[X,Y'] |G[X],G[Y]) + O(|\mc{P}|\log |Y|).$$ 	
 \end{obs}	
 
 Additionally we need a corollary of \cref{l:bases}.
 
\begin{cor}\label{c:homogenous2}
	Let $G$ be a graph on $n$ vertices, let $X,Y \subseteq V(G)$ be disjoint, and let $\mc{A}$ be a $(k,s)$-base on $Y$. Assume that $X$ is ordered, let $x =|X|$, and let $J$ be an $(x,k)$-ordered bigraph. Let $\mc{A}'$ be the set of all $A \in \mc{A}$ such that $G[X,A]$ is isomorphic to $J$. If $(kx+2)^2 2^{kx+2}|\mc{A}'| \leq s$  then \begin{equation}\label{e:homog}C_G(X,Y) \leq x|Y| +  C(\mc{A} | G[Y]) -  s/2^{kx+2}
	+ O(\log n).\end{equation}
\end{cor}

\begin{proof}
	
	Let $\mc{B}=\mc{A} - \mc{A}'$. Note that 	\begin{align}
 C(\mc{B}|G[Y]) &\leq  C(\mc{A} | G[Y]) +  C(\mc{A'}|\mc{A}) + O(\log n) \notag\\ 
	 &\leq C(\mc{A} | G[Y]) + |\mc{A}'|(\log|\mc{A}| - \log|\mc{A}'|) + 2|\mc{A}'| + O(\log n),\label{e:homog-1}
	\end{align}
	where the last inequality uses \cref{c:Kolmogorov2} (C3).
	Let $\alpha = kx+2$. Then $\alpha \geq 3$.
	As $\alpha^2 2^{\alpha}|\mc{A}'|~\leq~s=|\mc{A}|$ we have 
	\begin{align*}  |\mc{A}'|(\log|\mc{A}|  - \log|\mc{A}'|) + 2|\mc{A}'| \leq \frac{s(2+\log(\alpha^2 2^{\alpha}))}{\alpha^2 2^{\alpha}} \leq  \frac{s}{2^{\alpha}},
	\end{align*}  
	where the last inequality holds as $2+2 \alpha\log \alpha \leq \alpha^2$ for every $\alpha \geq 3$.
	Substituting this inequality into (\ref{e:homog-1}) we obtain
	$$ C(\mc{B}|G[Y]) \leq  \frac{s}{2^{{kx+2}}} + C(\mc{A} | G[Y]) + O(\log n).	$$
	Finally, combining this inequality with (\ref{e:bases1}),  we have 
	\begin{align*}C_G(X,Y) &\leq  x|Y|  + C(\mc{B}|G[Y]) - (kx - \log(2^{kx}-1))|\mc{B}| + O(\log n)
	 \\ &\leq  x|Y|  + C(\mc{A} | G[Y]) - (kx - \log(2^{kx}-1))|\mc{B}| + \frac{s}{2^{kx+2}}+ O(\log n) \\ & \leq x|Y|  + C(\mc{A} | G[Y]) - \frac{s}{2^{kx+1}} + \frac{s}{2^{kx+2}} + O(\log n),  \end{align*}
	implying (\ref{e:homog}), as desired. 
\end{proof}

Our next ingredient is an easy consequence of Ramsey's theorem. We say that $Z \subseteq V(G)$ is \emph{homogenous} (in $G$) if $Z$ is a clique or a stable set.

\begin{lem}\label{l:Ramsey} Let $r$ be an integer, and 
	let $G$ be a graph on $n$ vertices. Then there exist disjoint homogenous $r$ element subsets $Z_1,Z_2,\ldots Z_t$ of $V(G)$ such that $tr \geq n - 4^r$.
\end{lem}

\begin{proof} Let $Z_1,Z_2,\ldots Z_t$  be a maximal collection of disjoint homogenous $r$ element subsets of $V(G)$. Let $G' = G \setminus  (\cup_{i=1}^t Z_i)$. Then $|V(G')|= n - tr$. If $|V(G')| > 4^r$ then
	by Ramsey theorem $G'$ contains a homogenous subset of size $r$, in contradiction to the choice of $Z_1,\ldots Z_t$. Thus $|V(G')| \leq 4^r$, implying the desired inequality. 
\end{proof}

\begin{lem}\label{l:homogenous} For all $\eps,k >0$ there exist $\delta,n_0>0$ satisfying the following. Let $G$ be a graph with $|V(G)|=n \geq n_0$, and let $v\in V(G)$ be such that $C_G(v) \geq \eps n$. Then there exist disjoint homogenous subsets $Z_1,Z_2,\ldots,Z_m$ of $V(G)$ such that $m \geq \delta n$, and $v$ has at least $k$ neighbors and at least $k$ non-neighbors in $Z_i$.
\end{lem}

\begin{proof} Let  a positive integer $r$ be chosen so that $2k \binom{r}{k} \leq  2^{\eps r/4}$, and let $\delta >0 $ be such that $\delta(2-\log {\delta}) \leq \eps /4$ and $\delta r \leq \eps/4$. 

   By Lemma~\ref{l:Ramsey} there exists a (computable) collection disjoint homogenous subsets $\mc{Z}$ of $V(G)- \{ v \}$ such that $|Z|=r$ for every $Z \in \mc{Z}$, and $r|\mc{Z}| \geq n -1 - 4^r$. Let $\mc{B}$ be the collection of all sets $Z_i$ such that $v$ has  at least $k$ neighbors and at least $k$ non-neighbors in $Z_i$. Let $m = |\mc{B}|$. If   $m \geq \delta n$ then the lemma holds, and so we suppose that $m \leq \delta n$.
   
   By Corollary~\ref{c:Kolmogorov2}  the collection $\mc{A} = \mc{Z} - \mc{B}$ of the remaining subsets satisfies \begin{align}\label{e:hom11} C(\mc{A} | G \setminus v) &\leq m (\log n - \log m)+ 2m +  O(\log n) \notag\\ &\leq -\delta \log\delta n + 2\delta n + O(\log n) \leq \frac{\eps n}{4}+ O(\log n),\end{align}
   where the last inequality holds by the choice of $\delta$. 
Let $\mc{P}$ be the family of all bigraphs which have fewer than $k$ edges or fewer than $k$ non-edges. Note that $|\mc{A}| \leq n/r$  and $$|V(G)-\cup_{A \in \mc{A}}A| \leq mr +4^r+1.$$ Thus by Obervation~\ref{o:partition2} and (\ref{e:bases1}),  we have \begin{align}\label{e:hom12} C_G(v) &\leq C_G(\{v\},V(G)-\cup_{A \in \mc{A}}A)+C_G(\{v\},\cup_{A \in \mc{A}}A) + C(\mc{A} | G \setminus v) + O(\log n) \notag\\ &\leq mr + 4^r + 1 + C(\mc{A} | G \setminus v)  + h(1,r,\mc{P})\frac{n}{r} + O(\log n).\end{align}
Clearly, $h(1,r,\mc{P}) \leq \log ( 2k \binom{r}{k} ) \leq \eps r/4$, where the last inequality holds by the choice of $r$.  This inequality, together with (\ref{e:hom11}) and (\ref{e:hom12}) implies 
\begin{align*} C_G(v) \leq \delta nr  + 4^r + 1  + \frac{\eps n}{2}+ O(\log n) \leq \frac{ 3}{4}\eps n + 4^r + O(\log n). 
\end{align*}
Thus $C_G(v) < \eps n$, as long as $n$ is sufficiently large compared to $1/\eps$ and $r$, which is the desired contradiction. 
\end{proof}

\subsection{Typical structure of apex-free hereditary families}

With most of the ingredient in place we embark on the analysis of graphs in apex-free hereditary families.

The following lemma is the first step in our argument that uses the assumption that  $\mc{F}$ is apex-free.

\begin{lem}\label{l:vertexcritical} Let $\mc{F}$ be an apex-free hereditary family and let $l =\wpn(\mc{F})$. Then there exists $k >0$ such that for every $\eps>0$ there exist  $\delta,n_0 >0$ such for every $G \in \mc{F}$  with $|V(G)|=n \geq n_0$, every $(l,k,\delta n)$-perfect approximation $\mc{X}$ of $G$, and every $v \in V(G)$ we have $C_G(v, X - \{v\}) \leq \eps n$ for some $X \in \mc{X}$. 
\end{lem}

\begin{proof} By \cref{o:apexfree} there exists $k$ such that for all $0 \leq s \leq l$ there exists a graph $H(s) \not \in \mc{F}$ such that $|V(H(s))| \leq k$, $H(s) \setminus u \in \mc{H}(s,l-s)$ for some $u \in V(H)$. Let $\delta' \leq 1$ and $n_0'$ be chosen to satisfy Lemma~\ref{l:homogenous} for $\eps$ and $k$. We show that $\delta = \delta'/(6l^2)$  and $n_0= \max\{(l+1)n_0', 1/\delta\}$  satisfy the theorem.
	
Let $\mc{X} = (X_1,X_2,\ldots,X_l)$ be a $(l,k,\delta n)$-perfect approximation of $G$.
Suppose for a contradiction that $C_G(v,X_i) \geq \eps n$ for all $1 \leq i \leq l$. 	
 Let $m=\lceil \frac{\delta' n}{3l} \rceil$. By the choice of $\delta'$ for every $i \in [l]$ there exists a collection  $\mc{A}_i$ of $2k$ element subsets of $X_i$ such that $|\mc{A}_i|=m$ and either every set in $\mc{A}_i$ is a clique or every set induces a stable set, and $v$ has at least $k$ neighbors and at least $k$ non-neighbors in every set in $\mc{A}_i$. Without loss of generality we assume that the sets in $\mc{A}_i$ are stable sets for $i \leq s$ and cliques for $i >s$.

Let $H=H(s)$ be as in the first paragraph of the proof. We assume that $V(H)$ is ordered and let $\{u\},X_1,\ldots,X_l$ be a partition of $V(H)$ such that $X_i$ is a stable set for  $i \leq s$ and cliques for $i >s$. By restricting the sets in $\mc{A}_i$ and ordering them we can assume that $\mc{A}_i$ is a $(|X_i|,m)$-base and $G[v,A_i]$ and  $H[u,X_i]$ are isomorphic ordered bigraphs for all $A_i \in \mc{A}_i$. As $$m \geq \frac{\delta' n}{3 l} = 2l\delta n \geq l \lceil \delta n \rceil,$$ it follows that $\mc{X}$ is $(k,\lfloor m/l \rfloor)$-unrestricted. By Lemma~\ref{l:unrestricted} there exist $A_1 \in \mc{A}_1, A_2 \in \mc{A}_2 \ldots, A_l \in \mc{A}_l$ such that $G[A_i,A_j]$ and $H[X_i,X_j]$ are isomorphic for all $1 \leq i \leq j \leq l$.
Combining these isomorphisms we obtain an isomorphism between $G[\{v\} \cup A_1 \cup \ldots \cup A_l]$ and $H$, which is a contradiction as $G$ is $H$-free. 
\end{proof}

The final steps in the proof of Theorem~\ref{t:critical} require lower bounds on the growth of $\h(\mc{F},n)$.
For a family $\mc{F}$ and a positive integer $n$ define $$\Delta_k(\mc{F},n) = \min_{1 \leq d \leq k} \frac{\h(\mc{F},n) - \h(\mc{F},n-d) }{d}.$$
The following easy observation relates $\Delta_k(\mc{F},n)$ and conditional Kolmogorov complexity of small vertex subsets in a typical graph in $\mc{F}$.

\begin{obs}\label{o:smooth}
	Let  $\mc{F}$ be a hereditary family. Then for almost all $G \in \mc{F}$ with $|V(G)|=n$ for all $Z \subseteq V(G)$ we have \begin{equation}C_G(
	Z) \geq \Delta_{|Z|}(\mc{F},n)|Z| - \binom{|Z|}{2}- |Z|\log n - O(\log n) \label{e:smooth}\end{equation}
\end{obs}
\begin{proof} 
    The number of possible triples  of the form $(Z,G[Z],G[[n]-Z])$ where $G \in \mc{F}^n$ and  $Z \subseteq V(G)$ with $|Z| = k$ is bounded from above by $$2^{k \log n + \binom{k}{2} + \h(\mc{F},n-|Z|)+ O(\log n) } \leq 2^{\h(\mc{F},n)}2^{k\log n + \binom{k}{2} - \Delta_k(\mc{F},n)|Z| + O(\log n)}.$$
	 The inequality (\ref{e:smooth}) follows as  $G \in \mc{F}$  is uniquely determined by such a triple, $k$ and a binary string of length $C_G(Z)$.  
\end{proof}

The next lemma is the main remaining step in the proof of Theorem ~\ref{t:critical}.

\begin{lem}\label{l:supercritical}	Let $\mc{F}$ be an apex-free hereditary family, let $l=\wpn(G) \geq 2$.  Let $\mc{K} \subseteq \crit(\mc{F})$ be a finite set of  graphs. Then there exists an integer $k$ and $\delta>0$ such that for almost all $G \in \mc{F}$  on $n$ vertices, if $\Delta_k(\mc{F},n) \geq ((l-1)/l - \delta)n $ then there exists a $(\Forb(\mc{K}),l)$-partition of $V(G)$.   
\end{lem}

\begin{proof} Let an integer $k$ be chosen  so that 
	for all $0 \leq s \leq l-1$ and all $K \in \mc{K}$ there exists a graph $H(K,s)\not \in \mc{F}$ such that $|V(H)| \leq k$,  $H \setminus S \in \mc{H}(s,l-1-s)$ for some $S \subseteq V(H)$ such that $H[S]$ is isomorphic to $K$. Thus, in particular,  $|V(K)| \leq k$ for every $K \in \mc{K}$.  Let $\delta=\delta(l,k)>0$ be chosen  implicitly, sufficiently small to satisfy the inequalities appearing throughout the proof.

	By \cref{t:approx} there exists $\eps>0$ such that almost every $G \in \mc{F}$ with $|V(G)|=n$ admits a $(l,k,n^{1 - \eps})$-perfect approximation. Thus we assume that $G$ admits such an approximation $\mc{X}=(X_1,X_2,\ldots,X_l)$.
	
	 Let $Y_i$ denote the set of all vertices $v \in V(G)$ such that we have $C_G(v, X_i - Z) \leq \delta n$ for every $Z \subseteq V(G)$ with $v \in Z$, $|Z| \leq k$. By Lemma~\ref{l:vertexcritical} we have $\cup_{i=1}^l Y_i = V(G)$, as long as $n$ is sufficiently large. By Corollary~\ref{c:perfect2} we have $C(\mc{X}|_U \:| \: G[U]) = O(n^{1-\eps}\log n)$ for every $U \subseteq  V(G)$. Thus by Observation~\ref{o:partition2}, we have 
	 \begin{equation}\label{e:sup1}C_G(Z) \leq \sum_{i=1}^l C_G(Z, X_i - Z) + O(n^{1-\eps}\log n). \end{equation} 
	 for every $Z \subseteq V(G)$.
	 On the other hand, by Observation~\ref{o:smooth} we may assume that  \begin{equation}\label{e:sup2}C_G(Z) \geq \Delta_k(\mc{F},n)|Z| - O(\log n) \geq \left( \frac{l-1}{l} -\delta \right)|Z|n - O(\log n),\end{equation}  if $|Z| \leq k$, and \begin{equation}\label{e:sup3}  C_G(Z,X_i-Z) \leq |Z|\delta n+O(|Z|\log n) \end{equation} for every $Z \subseteq Y_i$. 
	 	
	 	 It follows that for every $Z \subseteq Y_i$ such that $|Z| \leq k$ and every $1 \leq j \leq l$, $j \neq i$ we have \begin{align} C_G&(Z,X_j-Z) \geq \notag\\ &\stackrel{(\ref{e:sup1})}{\geq} C_G(Z) -  C_G(Z, X_i - Z)  - \sum_{j' \in [l] -\{i,j\} }C_G(Z, X_{j'} - Z)- O(n^{1-\eps}\log n) \notag \\  &\stackrel{(\ref{e:sup2}),(\ref{e:sup3})}{\geq}  \left( \frac{l-1}{l} -\delta \right)|Z|n  - |Z|\delta n  - \sum_{j' \in [l] -\{i,j\} }C_G(Z, X_{j'} - Z)- O(n^{1-\eps}\log n)
	 	 \notag	\\
	 	 	 &\geq  \left( \frac{l-1}{l} -2\delta \right)|Z|n  - |Z|\sum_{j' \in [l] -\{i,j\} }|X_{j'}| - O(n^{1-\eps}\log n) \notag\\
	 	 	  &\geq  \frac{|Z|n}{l}- 2\delta|Z| n - O(n^{1-\eps}\log n). \label{e:highentropy}\end{align} In particular, the above inequality
 applied to a one element set $Z$ implies that $Y_i \cap Y_j = \emptyset$ for all $\{i,j\} \subseteq l$, and so  $(Y_1,\ldots,Y_l)$ is a partition of $V(G)$.  
 
 It remains to show that $G[Y_i]$ is $K$-free for all $K \in \mc{K}$.  Suppose for a contradiction. and without loss of generality,  that there exists $Z \subseteq Y_l$, $|Z| \leq k$ such that $G[Z]$ is isomorphic to $K$ for some $K \in \mc{K}$. Let $X'_i=X_i - Z$ for $1 \leq i \leq l-1$.  
 
 Let $m = \lceil lk 2^{k^2+3}\delta n\rceil$, and
 let $\mc{A}_i$ be a $(k,m)$-base in $G[X'_i]$ such that $C(\mc{A}_i | G[X'_i]) = O(1)$. Let $J$ be a $(|Z|,k)$-ordered bigraph, and let $\mc{A}'_i$ be the set of all $A \in \mc{A}_i$ such that $G[Z,A]$ is isomorphic to $J$. Then $|\mc{A}'_i| \geq \frac{m}{(k^2+2)^22^{k^2+2}}$, as otherwise, Corollary~\ref{c:homogenous2} implies $$C_G(Z,X'_i) \leq |Z||X'_i| - \frac{m}{2^{k^2+2}} + O(\log n),$$ contradicting (\ref{e:highentropy}) by the choice of $m$. 
 
 By Lemma~\ref{l:Ramsey} for every $1 \leq i \leq l-1$, as long as $n$ is large enough and $\delta$ is small enough, there exists a $(k,m)$-base $\mc{A}_i$  in $G[X'_i]$ consisting wither of stable sets, or  of cliques. In particular, some base with these properties is computable as a function of $G[X'_i]$ and $m$ ans so we may assume  $C(\mc{A}_i | G[X'_i]) = O(\log m)$.
  
  Suppose without loss of generality that $\mc{A}_i$ consists of stable sets for $i \leq s$ and of cliques for $s < i \leq l$. The graph $H=H(K,s) \not \in \mc{F}$ then has the property that $V(H)$ admits a partition $(S,A_1,\ldots,A_{l-1})$ such that $H[S]$ is isomorphic to $K$, $A_i$ is a stable set in $H$ for $i \leq s$ and $A_i$ is a clique for $i > s$. We assume that $V(H)$ is ordered.  Let $m'= \left \lceil  \frac{m}{(k^2+2)^22^{k^2+2}} \right \rceil$. As shown in the previous paragraph  $X_i'$ contains an $(|A_i|,m')$-base $\mc{A}'_i$ such that $G[Z,A']$ is isomorphic to $H[S,A_i]$ for every $A' \in \mc{A}'_i$. As in the proof of Lemma~\ref{l:vertexcritical}, using \cref{l:unrestricted} we can construct a subgraph of $G$ isomorphic to $H$, obtaining the desired contradiction.\end{proof}

The following technical lemma is the last remaining ingredient needed to derive~\cref{t:critical} from \cref{l:supercritical}. 

\begin{lem}\label{l:smooth}
	Let $\mc{F},\mc{F}_*$ be hereditary families such that  $\mc{F}_* \subseteq \mc{F}$ and $\mc{F}_*$ is smooth. Let $l =\wpn(\mc{F})$. Suppose that there exist an integer $k$ and $\delta>0$ such that if $\Delta_k(\mc{F},n) \geq ((l-1)/l - \delta)n $ then almost all graphs in $\mc{F}^n$ lie in $\mc{F}_*^n$. Then almost all  graphs in $\mc{F}$ lie in $\mc{F}_*$.
\end{lem}

\begin{proof}
	Let $f(n) = h(\mc{F},n)-((l-1)/l -\delta/2)n^2/2$.  Note that \begin{equation}\Delta_k(\mc{F},n) \geq \min_{1 \leq d \leq k}\frac{f(n) - f(n-d)}{d} +\left(\frac{l-1}{l}-\delta\right)n \label{e:deltak}\end{equation} for all sufficiently large $n$.  
	
	Let $I =\{ n \in \bb{N} | \Delta_k(\mc{F},n) \geq ((l-1)/l - \delta)n \}$. By Lemma~\ref{l:lower}, $\lim_{n \to \infty}f(n) = + \infty$. 	Thus (\ref{e:deltak}) implies that $I$ is infinite. 
	
	In particular, there exists $n_0 \in I$ such that $n_0 \geq 2k/\delta$,  $|\mc{F}^n_*| \geq |\mc{F}^n|/2$ for all $n \in I$ such that $n \geq n_0$, and $\h(\mc{F}_*,n+1) \geq  \h(\mc{F}_*,n) + ((l-1)/l -\delta/2)n$ for all $n \geq n_0$. These properties imply that for all  $0 \leq d \leq k$ and $n \in I$ such that $n \geq n_0$ we have
	\begin{align*} &\frac{\h(\mc{F},n+1) - \h(\mc{F},n+1-d)}{d} \\ &\geq  \frac{\h(\mc{F}_*,n+1) -   \h(\mc{F},n)}{d} +  \frac{\h(\mc{F},n) -  \h(\mc{F},n-(d-1))}{d} \\ &\geq  \frac{\h(\mc{F}_*,n+1) -   \h(\mc{F}_*,n) - 1}{d} + \frac{d-1}{d} \Delta_k(\mc{F},n) \\& \geq
	((l-1)/l - \delta)(n+1), 
	\end{align*}
	implying $n+1 \in I$. Thus every integer greater than or equal to $n_0$ belongs to $I$ and the lemma follows.
\end{proof}

\cref{t:critical} which we restate for convenience immediately follows from Lemmas~\cref{l:supercritical} and~\cref{l:smooth}. 

\critical*
\begin{proof} Let $\mc{F}_*=\mc{P}(\mc{T},l) \cap \mc{F}$. By \cref{l:supercritical} there exist $k$ and $\delta$ such that the conditions of
\cref{l:smooth} are satisfied. Therefore \cref{l:smooth} implies the theorem.
\end{proof}

The results of this section can also be conveniently used to give proofs Lemmas \ref{l:uniqueness1}  and  \ref{l:uniqueness2}, which have been postponed until now.
Both proofs are based on the following lemma. 

\begin{lem}\label{l:toUniqueness}
 Let $l \geq 2$ be an integer, let $\eps > 0$ be real, and let $\mc{T}$ be a meager hereditary family. Then there exists $\delta > 0$ and $n_0>0$ such that the following holds. Let $G$ be a graph on $n \geq n_0$ vertices such that $C(G) \geq \frac{l-1}{2l}n^2 - n^{2 - \eps}$  and  
 $C_G(v) \geq \left(\frac{l-1}{l} -\delta \right)n$ for every $v \in V(G)$. and $G$ admits a $(\mc{T},l)$-partition. Then the  $(\mc{T},l)$-partition is unique and $\eps$-balanced.  	
\end{lem}

\begin{proof}
	Let $k$ be such that $\mc{T} \subseteq \Forb(J)$ for some $(k,k)$-bigraph $J$. By \cref{l:balanced} the lower bound on $C(G)$ implies that there exists $\eps'>0$ such that  every  $(\mc{T},l)$-partition $\mc{X}$ of $G$ is $(l,k,n^{1-\eps'})$-perfect. By \cref{c:perfect2} this in turn implies that $C( \mc{X} | G \setminus v) = O(n^{1-\eps'}\log n)=o(n)$ for every $v \in V(G)$.\footnote{The constant hidden in $o(\cdot)$ notation might depend on $l$ and $\mc{T}$.}
	
	Suppose for a contradiction that $\mc{X}_1$ and  $\mc{X}_2$ are two distinct  $(\mc{T},l)$-partitions of $G$.  We claim that there exists $v \in V(G)$, $X_1 \in \mc{X}_1$ and $X_2 \in \mc{X}_2$ such that $v \in X_1 \cap X_2$ and $|X_1 \cup X_2| \geq \frac{n}{l} + \frac{n}{4l^2}$. Suppose not. Then $|X| \leq  \frac{n}{l} + \frac{n}{4l^2}$ for every $X \in \mc{X}_1 \cup \mc{X}_2$, and thus $$|X| \geq   n -  (l-1)\left(\frac{n}{l} + \frac{n}{4l^2} \right) \geq \frac{3}{4l}n$$ for every such $X$. Consider now $X_1 \in \mc{X}_1$ such that $X_1$ is not a subset of any part of $\mc{X}_2$. Then there exists   $X_2 \in \mc{X}_2$ such that $X_1 \cap X_2 \neq \emptyset$ and $|X_1 - X_2| \leq |X_1|/2$. Thus $$|X_1 \cup X_2| = |X_2| + |X_1 - X_2| \geq \frac{3}{4l}n  + \frac{3}{8l}n >  \frac{n}{l} + \frac{n}{4l^2},$$ a contradiction finishing the proof of the claim.
	
	Let $v,X_1,X_2$ be as in the  above claim. Note that $C_G[v,X_i - \{v\}]=o(n)$ 
by  \cref{l:homogenous}  for $i=1,2$ as $G[X_i] \in \mc{T}$ and $\mc{T}$ is meager. By considering the partition $(X_1 \cup X_2 \setminus \{v\}, V(G)-X_1-X_2)$ of $V(G)-\{v\}$ and using the bounds established above, we obtain \begin{align*}
	C_G(v) &\leq C_G[\{v\},V(G) - X_1 -X_2] \\ &+ \sum_{i=1}^{2}\left(C_G[\{v\},X_i- \{v\}] + C(X_i |  G \setminus v) \right) \\ &\leq |V(G)-X_1-X_2| + o(n) \\ &\leq \frac{l-1}{l}n  -\frac{n}{4l^2} + o(n).
	\end{align*}
	This is a contradiction, implying the lemma, as long as $\delta < \min\{1/4l^2, \eps'\}$ and $n$ is sufficiently large.
\end{proof}

We are now ready to finish the proofs of Lemmas \ref{l:uniqueness1}  and  \ref{l:uniqueness2} restated below.

\uniqone*

 \begin{proof} 
 		As observed in the proof of \cref{t:approx} we have $$ C(G) \geq \frac{(l-1)}{2l}n^2- O(n) $$ for almost all graphs $G \in \mc{F}$ on $n$ vertices and  thus there exist $\eps >0$ such that every 
 		$(\mc{T},l)$-partition of such a graph $G$ is $\eps$-balanced.

 		Suppose now that $\mc{F}$ is smooth. Then by \cref{o:smooth} for every $\delta > 0$ for almost all graphs $G \in \mc{F}$ on $n$ vertices, we have  $$C_G(v) \geq \left(\frac{l-1}{l} -\delta \right)n$$ for every $v \in V(G)$.  By \cref{l:toUniqueness} there exists $\delta > 0$ such that all sufficiently large graphs satisfying both of the above conditions satisfy the conclusion of \cref{l:uniqueness1}, as desired.
 \end{proof} 
 
\uniqtwo* 
 
  \begin{proof}
  	
  	Let $\mc{G}=(G_1,G_2,\ldots,G_l)$. As there exist $2^{\frac{l-1}{2l}n^2 - O(n^{2-\eps})}$ possible extensions of $\mc{G}$, we have $
	  C(G) \geq \frac{l-1}{2l}n^2 - O(n^{2-\eps})
  	$ for almost all such extensions. Similarly, for every $v \in [n]$ fixing the restriction $G'$ to $[n]-\{v\}$ of any extension of $\mc{G}$, there are $2^{\frac{l-1}{l}n - o(n)}$  ways to select the edges incident to $v$ and extend $G'$ to a complete extension of $v$. Thus for any $\delta >0$ the proportion of such extensions satisfying $C_G(v) \leq  \left(\frac{l-1}{l} - \delta\right)n$ is at most $2^{-\delta n + o(n)}$. By the union bound almost all extensions of  $\mc{G}$ satisfy $C_G(v) \geq  \left(\frac{l-1}{l} - \delta\right)n$ if $n$ is sufficiently large.
  	The desired conclusion now follows from \cref{l:toUniqueness}.
  \end{proof}
  
\section{Applications  of \cref{t:critical}}\label{s:applications}

The bulk of  this section is devoted to deriving \cref{t:critical2} from  \cref{t:critical}.  We also prove \cref{l:constellationspeed}. 

First, we will verify that the technical conditions in the statement \cref{t:critical} are satisfied for our intended application.

\begin{lem}\label{l:minnonstar} Let $s \geq 0$ be an integer, and let $G$ be a graph such that $G$ is not an $s$-star, but $G \setminus v$ is an $s$-star for every $v \in V(G)$. Then $|V(G)| \leq 4s+5$.
\end{lem}	

\begin{proof}
	Suppose for a contradiction that $|V(G)| \geq 4s+6$. Let $v \in V(G)$ be chosen arbitrarily. As $G \setminus v$ is an $s$-star, there exists a core $X$ of $G \setminus v$ such that $|X| \leq s$. Let $S = V(G) - X  - \{v\}$. Then $S$ is homogeneous, and $|S| \geq |V(G)| - s- 1 \geq 2$. We assume without loss of generality that $S$ is independent. Choose $u \in S$ such that, if $v$ is adjacent to at least $|S|/2$ vertices in $S$ then $uv \in E(G)$, and otherwise $uv \not \in E(G)$. Let  $X'$ be a core of $G \setminus u$ such that $|X'| \leq s$.
	
 Let $S' = V(G) - X'  - \{u\}$, then $|S \cap S'| \geq |V(G)|-2s-2 \geq 2$. Moreover,  every vertex in $V(G)-\{v\}$ is either adjacent to every vertex in $S$ or to no vertex in $S$, while every vertex in $V(G)-\{u\}$ is either adjacent to every vertex in $S'$ or to no vertex in $S'$, It follows that either $S \cup S'$ is independent in $G$, or $uv$ is the unique edge of $G[S \cup S']$. The second possibility, however, is eliminated by the choice of $u$.
	
	 As $G$ is not an $s$-star,  $S' \cup \{u\}$ is not a crown of $G$, that is there exists $v' \in V(G)$ and $w,w' \in S' \cup \{u\}$  such that $v'w \in E(G), v'w' \not \in E(G)$.  By the result of the previous paragraph we have $v' \neq u$.  Choose arbitrary $w'' \in S \cap S' - \{v'\}$. As $v'w \in E(G)$, we have $v'w'' \in E(G)$ unless $v'=v$ and $w=u$. Symmetrically, as $v'w' \not \in E(G)$, we have $v'w''  \not \in E(G)$ unless $v'=v$ and $w=u$. thu $v'=v$ and $u \in\{w,w'\}$.
	 
	 Suppose first that $u=w$. Then $v$ is adjacent to at least $|S|/2$ vertices in $S$, and so to at least one vertex of $S \cap S'$, as $|S \cap S'| > |V(G)|/2 \geq |S|/2$. It follows that $v$ is adjacent to every vertex in $S'$, contradicting existence of $w'$. The case $u=w'$ is analogous.
\end{proof}

\begin{cor}\label{c:generatestars} Every star-like hereditary family is finitely generated.
\end{cor}	

\begin{proof}
Let $\mc{H}$ be the collection of all graphs $H \not \in \mc{F}$ such that $H \setminus v \in \mc{F}$ for every $v \in V(H)$.	Then $\mc{F} = \Forb(\mc{H})$, and it suffices to show that $\mc{H}$ consists of finitely many isomorphism classes of graphs.

Let integers $s$ and  $n_0 \geq 4s+5$ be such that every graph in $\mc{F}^n$ is an $s$-star for every $n \geq n_0$.
By \cref{l:minnonstar} every $H \in \mc{H}$ with $|V(H)| \geq n_0$ is an $s$-star. By \cref{o:starsystems} this implies  there exists an  $s$-star system $\mc{J}$ such that $H \in \mc{P}(\mc{J})$. It is easy to see that for every star system $\mc{J}$ the family $\mc{P}(\mc{J})$ contains finitely many isomorphism classes of graphs in $\mc{H}$, and thus the corollary holds.
\end{proof}

The proof of the next useful lemma is analogous to the proof of \cref{l:cleanext} and we omit it.

\begin{lem}\label{l:starext} Let $\mc{J}$ be a constellation. Then $\mc{P}(\mc{J})$ is smooth.
\end{lem}	

We are now ready to prove a lemma which accomplishes most of the remaining technical work in the proof of \cref{t:critical2}.

\begin{lem}\label{l:starpartition} Let $l \geq 2$, $s \geq 0$ be integers, let $\mc{T}$ be an $s$-star-like hereditary family, and let  $\mc{F} \subseteq \mc{P}(\mc{T},l)$ be a hereditary family with $\chi_c(\mc{F})=l$.  Then  almost all  $G \in \mc{F}$ there exist an irreducible $(l,s)$-constellation $\mc{J}$ such that $G  \in \mc{P}(\mc{J})  \subseteq \mc{F}$.
\end{lem}

\begin{proof} Let $\mc{F}_* \subseteq \mc{F}$ denote the union of families $ \mc{P}(\mc{J})$ taken over all irreducible $(l,s)$-constellation $\mc{J}$ such that  $\mc{P}(\mc{J}) \subseteq \mc{F}$. Thus we need to show that almost every graph in $\mc{F}$ lies in $\mc{F}_*$.
	
By \cref{l:uniqueness1} there exists $\eps > 0$ such that almost every $G \in \mc{F}$ admits an $\eps$-balanced $(\mc{T},l)$-partition. Fix an $\eps$-balanced partition $\mc{X}=(X_1,X_2,\ldots,X_l)$ of $[n]$, and, as in Section~\ref{s:structureproof}, let $\h(\mc{X})$ be the number of pairs of elements of $[n]$ that do not belong to the same part  of $\mc{X}$. Let $\mc{G}=\mc{G}(\mc{X})$ denote the set of all graphs $G$ with $V(G)=[n]$ such that $\mc{X}$ is a  $(\mc{T},l)$-partition of $G$. As $\chi_c(\mc{F})=l$, we have $|\mc{G} \cap \mc{F}| \geq 2^{h(\mc{X})}$. 

We claim that\begin{itemize}
	\item[(P1)] $|\mc{G} \cap (\mc{F} -\mc{F}_*)| = o(2^{h(\mc{X})})$, and
	\item[(P2)] $\mc{X}$ is the unique  $(\mc{T},l)$-partition for almost every graph in $\mc{G} \cap \mc{F}_*$.
\end{itemize}
Clearly these two claims imply the lemma.

For the proof of (P1), let $\mc{F}(\mc{J},\psi)$ denote the set of graphs $G \in \mc{G} \cap \mc{F}$ such that $(\psi,\mc{X})$ is a  $\mc{J}$-template for $G$ for a fixed irreducible $(l,s)$-constellation $\mc{J}=(J,\phi,\alpha,\beta)$ and fixed $\psi: V(J) \to [n]$. Note that every graph in $\mc{G} \cap \mc{F}$ lies in $\mc{F}(\mc{J},\psi)$ for some choice of $\mc{J}$ and $\psi$ as above.

We upper bound $\mc{F}(\mc{J},\psi)$ for every $\mc{J}$ such that $\mc{P}(\mc{J}) \not \subseteq \mc{F}$. Let $H$ be such that $H \not \in \mc{F}$ and $H$ admits a $\mc{J}$-template $(\psi',Y_1,\ldots,Y_l)$. Let $Z$ denote the image of $\psi$, and let $Z'$ the image of $\psi'$.

Let $m = \lceil \frac{n}{2l|V(H)|}\rceil$.
For each $i \in [l]$, select pairwise disjoint $W^i_1, \ldots, W^i_m \subseteq X_i - Z$ such that $|W^i_j|=|Y_i-Z'|$ for every $j \in [m]$. Clearly such a choice is possible for large enough $n$.

Given $j \in [m]$ define $\eta_j: V(H) \to [n]$ to be a map such that  $\psi=\eta_j\psi'$ and $\eta_j(Y_i - Z')=W^i_j$ for every $i \in [l]$. Let $E_j$ be the set of all pairs of elements in the image of $\eta_j$ which do not both belong to $Z$ or to the same part of $\mc{X}$. Then $E_1,\ldots,E_m$ are pairwise disjoint by construction.

Moreover, for every $G \in \mc{F}(\mc{J},\psi)$ and every  $j \in [m]$, as  $\eta_j$ is not an embedding of $H$ into $G$, out of $2^{|E_j|}$ possible choices of subsets of pairs in $E_j$ at least one can not occur as $E(G) \cap E_j$. As the edges of $G$ on the parts of $\mc{X}$ are determined by the $\mc{J}$-pattern it follows that
$$|\mc{F}(\mc{J},\psi)| \leq 2^{h(\mc{X})} \prod_{j=1}^m\left(\frac{2^{|E_j|-1}}{2^{|E_j|}} \right) \leq  2^{h(\mc{X})} \left(1 - \frac{1}{2^{|V(H)|^2}} \right)^{\frac{n}{2l|V(H)|}}.$$
As the number of possible pairs $(\mc{J},\psi)$ is polynomial in $n$ this bound implies (P1).

For (P2) we again restrict our attention to a set $\mc{F}(\mc{J},\psi)$, but this time for an $(l,s)$-constellation $\mc{J}$ such that $\mc{P}(\mc{J})  \subseteq \mc{F}$. There exist graphs $G_1,G_2,\ldots,G_l \in \mc{T}$ such that $V(G_i)=X_i$ for every $i \in [l]$ and every $G \in \mc{F}(\mc{J},\psi)$ is an extension of $(G_1,G_2,\ldots,G_l)$. Moreover, $(\psi,\mc{X})$ is an $\mc{J}$-pattern for a constant fraction of all extensions of  $(G_1,G_2,\ldots,G_l)$ and all such extensions belong to  $\mc{F}(\mc{J},\psi)$. These observations together with \cref{l:uniqueness2} imply the second claim.
 \end{proof}	

We are now ready to derive the final conclusion of \cref{t:critical2} from \cref{t:critical}.

\begin{lem}\label{l:crit2} Let 
	$\mc{F}$ be an $(l,s)$-critical hereditary family then for almost all  $G \in \mc{F}$ there exist an $(l,s)$-constellation $\mc{J}$ such that $G  \in \mc{P}(\mc{J})  \subseteq \mc{F}$.
\end{lem}

\begin{proof} Define $\mc{F}_*$ as in the proof of \cref{l:starpartition}. Thus we need to show that almost every graph in $\mc{F}$ lies in $\mc{F}_*$.
	
Let $\mc{T}= \red(\mc{F})$. Then $\mc{T}$ is $s$-star-like, and in particular, meager, implying that $\mc{F}$ is apex-free. By \cref{c:generatestars} there exists a finite set of graphs $\mc{K} \subseteq \dang(\mc{F})$ such that $\mc{T} = \Forb(\mc{K})$. 

The family $\mc{F}_*$ is smooth by \cref{l:starext}. Meanwhile,  almost every graph in $\mc{P}(\mc{T},l) \cap \mc{F}$ lies in 	$\mc{F}_*$ by \cref{l:starpartition}. It follows that $\mc{P}(\mc{T},l) \cap \mc{F}$ is smooth. Therefore $\mc{T}$ satisfies all the requirements of \cref{t:critical}. Therefore by  \cref{t:critical}  almost every graph in $\mc{F}$ lies in $\mc{P}(\mc{T},l)$, and thus in $\mc{F}_*$.  
\end{proof}	

The following bound on the speed of hereditary families partitionable into meager components is the final ingredient of the proof of \cref{l:constellationspeed} and  \cref{t:critical2}.

\begin{lem}\label{l:decomposespeed}
	Let $\mc{T}_1, \mc{T}_2, \ldots, \mc{T}_l$ be extendable meager hereditary families,
	and let $\mc{F} = \mc{P}(\mc{T}_1, \mc{T}_2, \ldots, \mc{T}_l)$. Then there exists $\eps > 0$ such that 
\begin{equation}\label{e:decompose} \sum_{i=1}^l  \h(\mc{T}_i, n/l - n^{1-\eps}) - o(1) \leq  \h(\mc{F},n) -  \h(\mc{H}(l),n)  \leq  \sum_{i=1}^l  \h(\mc{T}_i, n/l + n^{1-\eps}) + o(1).
\end{equation}
\end{lem}	 
\begin{proof} Let $\mf{T}=(\mc{T}_1, \mc{T}_2, \ldots, \mc{T}_l)$. By \cref{c:uniqueness} there exists $\eps >0$ such that almost every graph in $\mc{F}$ admits a unique $\eps$-balanced $\mf{T}$-partition and  almost every graph in $\mc{H}(l)$ admits a unique $\eps$-balanced $(\mc{S},l)$-partition. (Recall, that $\mc{S}$ denote the family of edgeless graphs.) 
	 
Fix an $\eps$-balanced partition $\mc{X}=(X_1,\ldots,X_l)$ of $[n]$. Let $\mc{F}(\mc{X})$ denote the set of all graphs in $G \in \mc{F}^n$ such that  $\mc{X}$ is the unique  $\eps$-balanced $\mf{T}$-partition of $G$. Similarly, let $\mc{H}(\mc{X})$ denote the  set of all graphs in $G' \in \mc{H}^n(l)$  such that $\mc{X}$ is the unique  $\eps$-balanced $(\mc{S},l)$-partition of $G'$. 
By \cref{l:uniqueness2} we have  \begin{equation}\label{e:decompose1}|\mc{F}(\mc{X})| = (1 + o(1)) 2^{h(\mc{X})}\prod_{i=1}^l |\mc{T}_i^{|X_i|}|.\end{equation} 

Let $m(n) = \sum_{i=1}^l  h(\mc{T}_i, n/l - n^{1-\eps})$ and $M(n)= \sum_{i=1}^l  h(\mc{T}_i, n/l + n^{1-\eps})$. Clearly, we may assume that $n$ is sufficiently large, and thus $\h(\mc{T}_i, k)$ is a non-decreasing function of $k$ for $k \geq n/l - n^{1-\eps}$ and every $i \in [l]$. Thus  \begin{equation}\label{e:decompose2} 2^{m(n)} \leq  \prod_{i=1}^l |\mc{T}_i^{|X_i|}| \leq 2^{M(n)}.\end{equation}

\cref{l:uniqueness2} additionally implies that $|\mc{H}(\mc{X})| = (1 - o(1))2^{h(\mc{X})}$. Combining this estimate with (\ref{e:decompose1}) and (\ref{e:decompose2}) we obtain $$(1-o(1))2^{m(n)}|\mc{H}(\mc{X})| \leq |\mc{F}(\mc{X})|\leq (1+o(1))2^{M(n)}|\mc{H}(\mc{X})|.$$ 
Summing these inequalities over all $\eps$-balanced partitions  $\mc{X}$ of $[n]$ yields (\ref{e:decompose}) by the choice of $\eps$.
\end{proof}	

Having gathered the necessary ingredients, we finish this section with the proofs of \cref{l:constellationspeed} and \cref{t:critical2}, which are restated for convenience.

\constellationspeed*

\begin{proof}
	Let $\mf{F}=(\mc{P}(\mc{J}_i))_{i \in [l]}$ be the sequence of hereditary families corresponding to $s$-star systems within the constellation $\mc{J}$. It follows directly from the definition of $\mc{P}(\mc{J})$  that $\mc{P}(\mc{J}) \subseteq \mc{P}(\mf{F})$. Conversely, it is not hard to see (and was already observed in the proof of  \cref{l:starpartition}) that there exists $\eps> 0$ such that $|\mc{P}^n(\mc{J})| \geq \eps |\mc{P}^n(\mf{F})|$ for every positive integer $n$.
	
	Thus it remains to show that 
	\begin{equation}\label{e:decstar}  \h(\mc{P}(\mf{F}),n) -  \h(\mc{H}(l),n)  =|V(J)|\log n + O(1).
	\end{equation}
	By \cref{l:starspeed} we have $$ \h(\mc{P}(\mc{J}_i)), n) = |\phi^{-1}(i)|\log n + O(1),$$
	for every $i \in [l]$. Substituting the above identities into (\ref{e:decompose}) yields (\ref{e:decstar})
 \end{proof}	

\criticaltwo*

\begin{proof}
Clearly, \cref{l:crit2} implies the final statement of \cref{t:critical2}, and together with \cref{l:constellationspeed} this statement shows that the implication  \textbf{(i) $\Rightarrow$ (ii)} holds.

\vskip 5pt \textbf{(ii) $\Rightarrow$ (iii).}  Let $\mc{F}_1,\ldots,\mc{F}_l$	be as in (iii) and let $\mc{F}_*=\mc{P}(\mc{F}_1,\ldots,\mc{F}_l)$. It suffices to show that $\h(\mc{F}_*) -  \h(\mc{H}(l)) = \Omega(n)$. Suppose that $\mc{F}_1 \not \in \{\mc{C}^{+},\mc{S}^{+}\}$. Then  $\mc{F}_1$ is meager, and by \cref{l:decomposespeed}, it suffices to show that $\h(\mc{F}_1, n)=\Omega(n)$. This follows from \cref{t:BBW}, and is also not hard to verify independently.

It remains to consider the case $\mc{F}_1 \in \{\mc{C}^{+},\mc{S}^{+}\}$. By symmetry, we may assume  $\mc{F}_1 = \mc{C}^{+}$. Let $\mc{F}_{**} = \mc{P}(\mc{C},\mc{F}_2,\ldots,\mc{F}_l)$. Then $\mc{F}_*= (\mc{F}_{**})^{+}$. It follows that $\h(\mc{F}_*,n) \geq \h(\mc{F}_{**},n-1) + n-1.$
On the other hand, by \cref{l:decomposespeed} we have  $\h(\mc{F}_{**},n-1) = \h(\mc{H}(l),n-1)+O(1)$, and as $\mc{H}(l)$ is smooth we have $\h(\mc{H}(l),n) = \h(\mc{H}(l),n-1) + \frac{l-1}{l}n + o(n)$. It follows that $\h(\mc{F}_*,n) - \h(\mc{H}(l),n)\geq n/l - o(l),$ as desired.

\vskip 5pt \textbf{(iii) $\Rightarrow$ (i).} We refine the definition of $\red(\mc{F})$ as follows. Let  $0 \leq i \leq l-1$ be an integer. We denote by $\red(\mc{F},i)$ the family of all graphs $H$ such that  $\mc{P}(\iota(H),\mc{H}(i,l-1-i)) \subseteq \mc{F}$. Clearly, $\red(\mc{F},i)$ is hereditary, and $\red(\mc{F}) = \cup_{i=0}^{l-1}\red(\mc{F},i)$. 

Suppose that (i) does not hold, that is $\mc{F}$ is non-critical. Then $\red(\mc{F})$ is non-star-like. Thus $\red(\mc{F},i)$ is non-star-like for some $0 \leq i \leq l-1$. By \cref{t:BBW} there exists $\mc{F}_1 \subseteq \red(\mc{F},i)$ such that either $\mc{F}_1$ or $\bar{\mc{F}}_1$ contains one of the families $\mc{C} \vee \mc{C}, \mc{C} \vee \mc{S}, \mc{M}$, $\mc{C}^{+}$. Let $\mc{F}_j=\mc{S}$ for $2 \leq j \leq i+1$, and let $\mc{F}_j=\mc{C}$ for $i+2 \leq j \leq l$. Then $\mc{P}(\mc{F}_1,\ldots,\mc{F}_l) \subseteq \mc{F}$ by definition of $\red(\mc{F},i)$. Therefore (iii) does not hold, as desired.
\end{proof}	  

\subsubsection*{Acknowledgement.} We are indebted to Bruce Reed for introducing us to the questions considered and to many of the techniques employed in this paper.  We thank Zachary Feng for valuable comments.

\bibliographystyle{alpha}
\bibliography{../snorin}

\end{document}